\documentclass[english,11pt]{article}

\usepackage[francais]{babel}
\usepackage[T1]{fontenc}
\usepackage[latin1]{inputenc}
\usepackage{hyperref}
\usepackage{lmodern}
\usepackage{amsmath}
\usepackage{amsfonts}
\usepackage{amsthm}
\usepackage{amssymb}
\usepackage{geometry}
\usepackage{graphicx}
\usepackage[all]{xy}
\usepackage{faktor}
\newtheorem*{rem}{Remark}

\newtheorem{theo}{Theorem}[section]
\newtheorem*{theo*}{Theorem}
\newtheorem{prop}[theo]{Property}
\newtheorem{defi}[theo]{Definition}

\newtheorem{lem}[theo]{Lemma}

\newcommand{\Dz}{\mathcal{D}^{(0)}}
\newcommand{\Dhz}{\hat{\mathcal{D}}^{(0)}}
\newcommand{\Dh}{\hat{\mathcal{D}}}
\newcommand{\Ddag}{{\mathcal{D}}^\dagger}
\newcommand{\Q}{\mathbb{Q}}
\newcommand{\calS}{\mathcal{S}}
\newcommand{\calT}{\mathcal{T}}
\newcommand{\calX}{\mathcal{X}}
\newcommand{\calY}{\mathcal{Y}}
\newcommand{\calZ}{\mathcal{Z}}
\newcommand{\calA}{\mathcal{A}}
\newcommand{\calB}{\mathcal{B}}
\newcommand{\calD}{\mathcal{D}}
\newcommand{\calO}{\mathcal{O}}
\newcommand{\calE}{\mathcal{E}}
\newcommand{\calF}{\mathcal{F}}
\newcommand{\calL}{\mathcal{L}}
\newcommand{\calM}{\mathcal{M}}
\newcommand{\calP}{\mathcal{P}}
\newcommand{\calQ}{\mathcal{Q}}
\newcommand{\calR}{\mathcal{R}}

\title{Fourier-Mukaï transform for $\Dhz$-modules over formal abelian schemes}
\author{Florian \bsc{Viguier}}
\date{\today}

\begin{document}

\maketitle

\selectlanguage{english}
\begin{abstract}
%Let $A$ be an abelian variety over $S$, a localy noetherian scheme. Fourier-Mukai transform, introduced in 1981 by Mukai for $S$ the spectrum of an algebraic close field, allow to establish an equivalence of categories between quasi-coherent sheaves over $A$ and the ones over $A^\vee$, its dual variety. Later, these results have been established in one of Laumon's preprint for an abelian variety over a base $S$ locally noetherian.
In 1996, Rothstein and Laumon simultaneously constructed a Fourier-Mukai transform for $\calD$-modules over a locally noetherian base of characteristic 0. This functor induces an equivalence of categories between quasi-coherent sheaves of $\calD$-modules over an abelian variety $A$ and quasi-coherent sheaves of $\calO$-modules over its universal vectorial extension $A^\natural$.
In this article, we define a Fourier-Mukai transform for $\calD$-modules on an abelian formal scheme $\calA/\calS=Spf(V)$, where $V$ is a discrete valuation ring, and we discuss the extension of the classical results of Fourier-Mukai transform to this arithmetic case.
\end{abstract}

\tableofcontents

\section*{Introduction}

The Fourier-Mukai transform, first defined by Shigeru Mukai in \cite{Muk} for $\calO$-modules on an abelian variety over an algebraically closed field, is a very usefull tool in the study of quasi-coherent sheaves over an abelian variety. Thanks to the independant work of Rothstein \cite{Rot96} and Laumon \cite{Lau}, it has been extended into a functor acting on $\calD$-modules on an abelian variety in characteristic zero. We will here recall the way it is constructed in \cite{Lau}, as it will be the easier construction to adapt for arithmetic varieties.

First, given an abelian variety $A$ over a base $S$ of characteristic zero, one constructs the functor from the category of noetherian $S$-scheme to the one of abelian groups:
\[Pic^\natural(A\times\bullet/\bullet):T\mapsto Pic^\natural(A\times T/T),\]
where $Pic^\natural(X/T)$ is the abelian group of sheaves of $\calD_{X/T}$-modules $\calO_X$-invertible $\calL$ such that
\[m^\flat\calL\simeq p_1^\flat\mathcal{L}\otimes p_2^\flat\mathcal{L},\]
with $m,p_1,p_2:X\times_T X\rightarrow X$ being respectively the multiplication and the canonical projections and $f^\flat$ the inverse image functor of $\calD$-modules without the shift. Such a sheaf $\calL$ is said to verify the theorem of the square. Note that the group $Pic^\natural(A/S)$ is also the group of rigid extensions of $A$ (see \cite{Ma-Me} for more details).

\noindent The functor $Pic^\natural(A\times\bullet/\bullet)$ is representable by an abelian group scheme over $S$ of dimension $2dim(A)$ and denoted $A^\natural$. It is the universal vectorial extension of $A$. The representability of this functor leads to the construction of a universal element $\calP$ of $Pic^\natural(A\times A^\natural/A^\natural)$, corresponding to the morphism $id_{A^\natural}\in Hom(A^\natural,A^\natural)$. The sheaf $\calP$ is thus a $\calD_{A\times A^\natural/A^\natural}$-module, $\calO_{A\times A^\natural}$-invertible, satisfying the theorem of the square, and it is called the Poincar\'e sheaf. Note that it is not the same Poincar\'e sheaf as the one defined on the dual abelian variety. These sheaves are linked together in a way that will be explained later.

\noindent The classical Fourier-Mukai transform is then defined as the functor $\calF$ from the derived category of quasi-coherent $\calD_{A/S}$-modules to the derived category of quasi-coherent $\calO_{A^\natural}$-modules as follows
\[\mathcal{F}(\mathcal{E}^\cdot)=Rp_*^\natural(\mathcal{P}\overset{\mathbb{L}}{\otimes}p^*\mathcal{E}^\cdot),\]
where $p:A\times A^\natural \rightarrow A$ and $p^\vee: A\times A^\natural \rightarrow A^\natural$ are the canonical projections. One can define in the same way the dual Fourier-Mukai transform $\mathcal{F}^\natural:D^b_{qcoh}(\mathcal{O}_{A^\natural})\rightarrow D^b_{qcoh}(\mathcal{D}_{A/S})$ as follows
\[\mathcal{F}^\natural(\mathcal{E}^\cdot)=Rp_*(\mathcal{P}\overset{\mathbb{L}}{\otimes}p^{\natural*}\mathcal{E}^\cdot).\]

The most important property of the Fourier-Mukai transform for $\calD$-modules is the following

\begin{theo*}[Rothstein-Laumon]

The Fourier-Mukai transform induces an equivalence of categories between $D^b_{qcoh}(\calD_{A/S})$ and $D^b_{qcoh}(\calO_{A^\natural})$.

More precisely, it is involutive, in the sense that
\[\calF^\natural\circ\calF\simeq\langle-1\rangle^!\bullet[-dim(A)]\]
and
\[\calF\circ\calF^\natural\simeq\langle-1\rangle^{\natural*}\bullet[-dim(A)],\]
with $\langle-1\rangle$ and $\langle-1\rangle^\natural$ being the inverse morphisms of $A$ and $A^\natural$ respectively.

\end{theo*}

Roughly speaking, a consequence of this theorem is that a complex $\calE^\cdot$ of $\calD$-modules over $A$ is characterized by the solutions of all the twists of $\calE^\cdot$ with invertible sheaves with connexions.

\noindent The goal of any extension of the Fourier-Mukai transform is to preserve this property.

One of the motivations of this work comes from an article of Schnell (\cite{Sch15}) where he uses the Fourier-Mukai transform to establish some structure results for holonomic $\calD$-modules.

In the arithmetic case we fix $V$ a discrete valuation ring of uniformizer $\pi$ with $p=char(V/\pi)>0$, $V_i$ the reduction modulo $\pi^i$ of $V$, $(A_i)$ an inductive system of abelian schemes over the $S_i=Spec(V_i)$ which are compatible with the reductions modulo $\pi^j$ (e.g. the reduction modulo $\pi^i$ of an abelian variety $A$ over $S$) and $\calA=\varinjlim A_i$, which is a formal abelian scheme.

\noindent The goal of this article is to define a Fourier-Mukai transform on $\calA$ and to study it in order to find the expected involutivity formulas, as it has been previously done for $\calO_\calA$-modules in \cite{Vig22}. Note that all the previous constructions will adapt without any problem in the arithmetic case, the only work will be to define the functors correctly. However, the proof of the involutivity formulas strongly requires the characteristic zero hypothesis, so that a lot of work is needed in order to prove results in the arithmetic case.
The main result of this article is theorem \ref{Dhz_equiv} where we prove that our arithmetic Fourier-Mukai functor is essentially surjective.

The first chapter will be dedicated to the definitions of "relative" $\calD$-modules and their fundamental properties. Indeed, as $\calP$ is a $\calD_{A\times A^\natural/A^\natural}$-module, all the $\calD$-modules will be considered as relative to $A^\natural$. It is important to study their behavior, especially their push-forward by the projection $A\times A^\natural\rightarrow A$.

\noindent These discussions will also lead to an interesting version of the seesaw principle for $\calD$-modules (property \ref{D-bascule}).

The second chapter will focus on the definition of the dual abelian variety $\calA^\natural$, the Poincar\'e sheaf $\calP$ and the Fourier-Mukai transform $\calF$ as well as their properties, the most important being the essential surjectivity of $\calF$ (theorem \ref{Dhz_equiv}).

I want to thank Christine Huyghe for the numerous conversations we had on this topic and for all of her precious advices.

\section{Classic and relative $\Dhz$-modules}

For this section, let $V$ be a discrete valuation ring with uniformizer $\pi$, $S=Spec(V)$, for all $i$ $V_i=\faktor{V}{\pi^i V}$ and $S_i=Spec{V_i}$ and $\calS=Spf{S}$. For all $i$ let $X_i$ be a variety over $S_i$ such that $X_{i+1}\times_{S_{i+1}}S_i\simeq X_i$ and $\calX=\varinjlim{X_i}$ is smooth (one can think of the $X_i$ as the reduction modulo $\pi^i$ of a smooth variety $X$ over $S$).

By the end of this article, the $X_i$ will be considered to be abelian schemes, but as the following definitions can be written in a more general context, we will only assume $\calX$ is smooth.

\subsection{Definitions}

The construction of the sheaves $\Dz$ won't be recalled in this article. For more information, see \cite{D_mod1} (chapter 1).
Simply recall that taking a left $\Dz$-module structure on a $\calO$-module $\calE$ is equivalent to take an integrable connexion on $\calE$ (this is the main reason this article will only talk about $\Dz$-modules and not $\mathcal{D}^{(m)}$-modules). Moreover, locally
\[\Gamma(U,\Dz_X)=\left\{\sum_{\underline{k}\text{, finite}} a_{\underline{k}} \underline{\partial}^{\underline{k}}\mid a_{\underline{k}}\in\Gamma(U,\mathcal{O}_X)\right\},\]
where the $\partial_i$ are the derivation associated to a coordinate system $t_i$ of $U$ (note that these coordinates, and so $\Dz_X$, depend on the base $S$ considered).

One can define the module $\Dz_\mathcal{X}$ the same way on the formal scheme $\calX$.

\begin{defi}

Let $\mathcal{X}$ be a formal scheme on $\mathcal{S}$. $\Dhz_{\mathcal{X}/\calS}$ is the $\pi$-adic completion of $\Dz_\mathcal{X}$:
\[\Dhz_{\mathcal{X}/\calS}=\underset{i}{\varprojlim}\faktor{\Dz_\mathcal{X}}{\pi^i}\simeq \underset{i}{\varprojlim}\Dz_{X_i}.\]

\end{defi}

\begin{rem}
	The subscript $\calX/\calS$ is here to point out that $\calX$ is viewed as a $\calS$-formal scheme. Later in this paper we will have to consider the product of two formal schemes $\calX\times\calY$ as a scheme over $\calY$, and so we will have to consider $\Dhz_{\calX\times\calY/\calY}$-modules instead of $\Dhz_{\calX\times\calY/\calS}$-modules. In that case, we will call them 'relative' $\Dhz$-modules.
	
	In the following we may however use the notation $\Dhz_{\mathcal{X}}$ instead of $\Dhz_{\mathcal{X}/\calS}$ if there is no confusion.
\end{rem}

Locally,
\[\Gamma(\mathcal{U},\Dhz_{\mathcal{X}/\calS})=\left\{\sum_{\underline{k}} a_{\underline{k}} \underline{\partial}^{\underline{k}}\mid a_{\underline{k}}\in\Gamma(U,\mathcal{O}_\calX)\text{ and } \underset{|\underline{k}|\rightarrow\infty}{v_\pi(a_{\underline{k}})\rightarrow0}\right\}.\]

\subsubsection*{Quasi-coherent $\Dhz$-modules}

We will take the following definition for quasi-coherent sheaves of $\Dhz$-modules over $\calX$:

\begin{defi}

Let $\mathcal{E}^\cdot\in D^-(\Dhz_\mathcal{X})$. For all $i$ define $\mathcal{E}^\cdot_i=\mathcal{O}_{X_i}\overset{\mathbb{L}}{\otimes}_{\mathcal{O}_\mathcal{X}}\mathcal{E}^\cdot$.

$\mathcal{E}^\cdot$ is said quasi-coherent if
\begin{itemize}
\item $\mathcal{E}^\cdot_0\in D^-_{qcoh}(\Dz_{X_0})$,
\item There is a canonical isomorphism $\mathcal{E}^\cdot\simeq R\varprojlim\mathcal{E}^\cdot_i$.
\end{itemize}

Denote by $D^-_{qcoh}(\Dhz_\mathcal{X})$ the full sub-category of $D^-(\Dhz_\mathcal{X})$ whose objects are the quasi-coherents.

\end{defi}

The following characterization can be find in \cite{D_mod1}.

\begin{theo}[Berthelot]

If $(\mathcal{E}_i^\cdot)_i$ is a projective system of complexes such that the $i$-th element is in $D^b(\Dz_{X_i})$ and verifying
\begin{itemize}
\item $\mathcal{E}^\cdot_0\in D^b_{qcoh}(\Dz_{X_0})$,
\item For all $i$ the canonical morphism
\[\Dz_{X_i}\overset{\mathbb{L}}{\otimes}_{\Dz_{X_{i+1}}}\mathcal{E}^\cdot_{i+1}\rightarrow\mathcal{E}^\cdot_i\]
is an isomorphism,
\end{itemize}
then $\mathcal{E}^\cdot=R\varprojlim\mathcal{E}^\cdot_i$ is an object of $D^b_{qcoh}(\Dhz_\mathcal{X})$.

\end{theo}

\subsubsection*{Tensor product}

We will only recall the definition given in \cite{D_mod1} (chapter 3) and some basic properties.

\begin{defi}
Let $\mathcal{E}^\cdot\in D^-_{qcoh}(\Dhz_\mathcal{X}\ ^r)$ (resp. $\mathcal{F}^\cdot\in D^-_{qcoh}( ^l\Dhz_\mathcal{X})$) be a complex of right (resp. left) $\Dhz_\mathcal{X}$-module.

\noindent As usual, denote $\mathcal{E}^\cdot_i=\mathcal{O}_{X_i}\overset{\mathbb{L}}{\otimes}_{\mathcal{O}_\mathcal{X}}\mathcal{E}^\cdot$. Idem for $\mathcal{F}^\cdot_i$.

The completed tensor product of $\mathcal{E}^\cdot$ and $\mathcal{F}^\cdot$ is defined as
\[\mathcal{E}^\cdot\hat{\otimes}^\mathbb{L}_{\Dhz_\mathcal{X}}\mathcal{F}^\cdot=R\underset{i}{\varprojlim}\left(\mathcal{E}^\cdot_i\overset{\mathbb{L}}{\otimes}_{\Dz_{X_i}}\mathcal{F}^\cdot_i\right).\]

\end{defi}

It is a $\Dhz_\mathcal{X}$-module.

Moreover, if the complexes have bounded cohomology, $\mathcal{E}^\cdot\hat{\otimes}^\mathbb{L}_{\Dhz_\mathcal{X}}\mathcal{F}^\cdot\in D^b_{qcoh}(\Dhz_\mathcal{X})$.

Finally, if one of the complexes is coherent, there is a canonical isomorphism \[\mathcal{E}^\cdot\overset{\mathbb{L}}{\otimes}_{\Dhz_\mathcal{X}}\mathcal{F}^\cdot\simeq\mathcal{E}^\cdot\hat{\otimes}^\mathbb{L}_{\Dhz_\mathcal{X}}\mathcal{F}^\cdot.\]

\subsubsection*{Inverse image}

The following definition can be found in \cite{D_mod1} (chapter 3).

\begin{defi}

Let $f:\mathcal{X}\rightarrow\mathcal{Y}$ be a morphism of smooth formal $\mathcal{S}$-schemes and $\mathcal{E}^\cdot\in D^b_{qcoh}(\Dhz_\mathcal{Y})$.

\noindent For all $i$, denote by $f_i:X_i\rightarrow Y_i$ the morphism of $S_i$-schemes associated to $f$.

The inverse image of $\mathcal{E}^\cdot$ by $f$ is defined as

\[f^!_{/\calS}\mathcal{E}^\cdot = R\underset{i}{\varprojlim}\left(f^!_{i/S_i}\mathcal{E}^\cdot_i\right).\]

\end{defi}

\begin{rem}

The notation $f^!_{/\calS}$ is used here to show explicitly the base $\calS$. This notation will make sense in the second half of this article, where there will be schemes which can be seen as $\calS$-schemes or $\mathcal{T}$-schemes (with $\mathcal{T}$ as $\calS$-scheme). The base scheme is of crucial importance when it comes to $\calD$-modules, as it determines which functions are the constants, and the functors $f^!_{/\calS}$ and $f^!_{/\mathcal{T}}$ are different.

\noindent However, if there is no confusion possible, we will use the usual notation $f^!$.

\end{rem}

Again, if $\mathcal{E}^\cdot$ is coherent, there is a canonical isomorphism

\[\Dhz_{\mathcal{X}\rightarrow\mathcal{Y}}\overset{\mathbb{L}}{\otimes}_{f^{-1}\Dhz_\mathcal{Y}}f^{-1}\mathcal{E}^\cdot[d_\mathcal{X}-d_\mathcal{Y}]\simeq f^!\mathcal{E}^\cdot,\]
where the transfer bimodule is defined as $\Dhz_{\mathcal{X}\rightarrow\mathcal{Y}}=\varprojlim\Dz_{X_i\rightarrow Y_i}$.

The formal functor $f^!$ preserves quasi-coherence and, if $f$ is smooth, preserves coherence. This is because the 'classical' functors $f_i^!$ do and (quasi-)coherence can be tested modulo $\pi^i$.

As relative $\Dhz$-modules will be crucial in the article, it is necessary to define the relative inverse image. Let $\calX$, $\calY$ and $\calZ$ be three formal schemes over $\calS$, $g:\calY\rightarrow\calZ$ and $f=id_\calX\times g:\calX\times\calY\rightarrow\calX\times\calZ$ and assume that $\calE^\cdot\in D^b_{qcoh}(\Dhz_{\calX\times\calZ/\calZ})$ (so only the derivations along $\calX$ acts on $\calE^\cdot$). The goal is to give some sense to the notion of pulling-back $\calE^\cdot$ on $\calX\times\calY$. The main idea is to say that we want to preserve the action of the derivations along $\calX$ and keep having no action of the derivations along $\calY$. It can be thought as a base change.

\begin{defi}

	For each $i$, define the transfer bimodule (of level 0) relative to $Y_i$ and $Z_i$ as
		\[\mathcal{D}^{(0)(Y_i,Z_i)}_{X_i\times Y_i\rightarrow X_i\times Z_i/S_i}=f_i^*\Dz_{X_i\times Z_i/Z_i}=\mathcal{O}_{X_i\times Y_i}\otimes_{f_i^{-1}\mathcal{O}_{X_i\times Z_i}}f_i^{-1}\Dz_{X_i\times Z_i/Z_i}.\]
	
	The transfer bimodule (of level 0) relative to $\calY$ and $\calZ$ can then be defined as
	\[\Dh^{(0)(\mathcal{Y},\mathcal{Z})}_{\mathcal{X}\times\mathcal{Y}\rightarrow\mathcal{X}\times\mathcal{Z}/\calS}=\underset{i}{\varprojlim}\mathcal{D}^{(0)(Y_i,Z_i)}_{X_i\times Y_i\rightarrow X_i\times Z_i}.\]
	
	For all $\mathcal{E}_i^\cdot\in D^b(\Dz_{X_i\times Z_i/Z_i})$, the inverse image relative to $Y_i$ and $Z_i$ of $\calE^\cdot_i$ by $f_i$ is defined by
	\[f_{i/S_i}^{!(Y_i,Z_i)}\mathcal{E}_i^\cdot=\mathcal{D}^{(0)(Y_i,Z_i)}_{X_i\times Y_i\rightarrow X_i\times Z_i}\overset{\mathbb{L}}{\otimes}_{f_i^{-1}\Dz_{X_i\times Z_i/Z_i}}f_i^{-1}\mathcal{E}_i^\cdot\in D^b(\Dz_{X_i\times Y_i/Y_i}).\]
	
	Finally, for all $\mathcal{E}^\cdot\in D^b_{qcoh}(\Dhz_{\mathcal{X}\times\calZ/\mathcal{Z}})$, the inverse image relative to $\mathcal{Y}$ and $\mathcal{Z}$ of $\calE^\cdot$ by $f$ is defined by
	\[f^{!(\mathcal{Y},\mathcal{Z})}_{/\calS}\mathcal{E}^\cdot = R\underset{i}{\varprojlim}\left(f_i^{!(Y_i,Z_i)}\mathcal{E}^\cdot_i\right).\]

\end{defi}

These are exactly the same definitions as in the non-relative case, at the only exception that all the subscripts $/S$ have been changed into subscripts $/Y$ or $/Z$, as now the variety $X\times Y$ (resp. $X\times Z$) is viewed as a $Y$-variety (resp. $Z$-variety). Also note that because of that the natural shift that should occur in the definition of the inverse image is $dim_Y(X\times Y)-dim_Z(X\times Z)=dim(X)-dim(X)=0$. This explains why there is no shift in the definition.

It is important to point out that these relative functors are actually very close to the functors defined on $\mathcal{O}$-modules, as they don't modify the structure of $\Dhz$-module of the sheaves they act on. Keeping that in mind, the following results are straightforward.

\begin{prop}

\begin{itemize}

\item $f_i^{!(Y_i,Z_i)}$ and $f^{!(\mathcal{Y},\mathcal{Z})}$ verify the transitivity formulas.

\item $f_i^{!(Y_i,Z_i)}\Dz_{X_i\times Z_i/Z_i}\simeq\Dz_{X_i\times Y_i/Y_i}$.

\item $f^{!(\mathcal{Y},\mathcal{Z})}\Dhz_{\mathcal{X}\times\mathcal{Z}/\mathcal{Z}}\simeq\Dhz_{\mathcal{X}\times\mathcal{Y}/\mathcal{Y}}$.

\item As $\mathcal{O}_{X_i\times Y_i}$-modules, $f_i^{!(Y_i,Z_i)}\calE_i^\cdot\simeq Lf_i^*\calE_i^\cdot$.

\item As $\mathcal{O}_{\calX\times \calY}$-modules, $f^{!(\calY,\calZ)}\calE^\cdot\simeq Lf^*\calE^\cdot$, where $Lf^*$ is defined as $Lf^*\calE^\cdot=R\varprojlim Lf_i^*\calE_i^\cdot$.

\item If $\mathcal{E}_i^\cdot\in D^b_{qcoh}(\Dz_{X_i\times Z_i/Z_i})$, then $f_i^{!(Y_i,Z_i)}\mathcal{E}_i^\cdot\in D^b_{qcoh}(\Dz_{X_i\times Y_i/Y_i})$.

\item If $\mathcal{E}^\cdot\in D^b_{qcoh}(\Dhz_{\mathcal{X}\times\mathcal{Z}/\mathcal{Z}})$, then $f^{!(\mathcal{Y},\mathcal{Z})}\mathcal{E}^\cdot\in D^b_{qcoh}(\Dhz_{\mathcal{X}\times\mathcal{Y}/\mathcal{Y}})$.

\item If $\mathcal{E}^\cdot\in D^b_{coh}(\Dhz_{\mathcal{X}\times\mathcal{Z}/\mathcal{Z}})$ there is a canonical isomorphism

\[\Dh^{(0)(\mathcal{Y},\mathcal{Z})}_{\mathcal{X}\times\mathcal{Y}\rightarrow\mathcal{X}\times\mathcal{Z}}\overset{\mathbb{L}}{\otimes}_{f^{-1}\Dhz_{\mathcal{X}\times\mathcal{Z}/\mathcal{Z}}}f^{-1}\mathcal{E}^\cdot\simeq f^{!(\mathcal{Y},\mathcal{Z})}\mathcal{E}^\cdot.\]

\item If $f_i$ (resp. $f$) is smooth, then $f_i^{!(Y_i,Z_i)}$ (resp. $f^{!(\mathcal{Y},\mathcal{Z})}$) preserves coherence.

\end{itemize}

\end{prop}

\subsubsection*{Direct image}

As for inverse images, the first defintion can be found in \cite{D_mod1} (chapter 3).

\begin{defi}

Let $f:\mathcal{X}\rightarrow\mathcal{Y}$ be a morphism of smooth formal $\mathcal{S}$-schemes and $\mathcal{E}^\cdot\in D^b_{qcoh}(\Dhz_\mathcal{X})$.

\noindent For all $i$, denote by $f_i:X_i\rightarrow Y_i$ the morphism of $S_i$-schemes associated to $f$.

The direct image of $\mathcal{E}^\cdot$ by $f$ is defined as

\[f_{+/\calS}\mathcal{E}^\cdot = R\underset{i}{\varprojlim}\left(f_{i+/S_i}\mathcal{E}^\cdot_i\right).\]

\end{defi}

Again, if $\mathcal{E}^\cdot$ is coherent there is a canonical isomorphism

\[Rf_*\left(\Dhz_{\mathcal{Y}\leftarrow\mathcal{X}}\overset{\mathbb{L}}{\otimes}_{\Dhz_\mathcal{X}}\mathcal{E}^\cdot\right)\simeq f_+\mathcal{E}^\cdot,\]
where the transfer bimodule is defined as $\Dhz_{\mathcal{Y}\leftarrow\mathcal{X}}=\varprojlim\Dz_{Y_i\leftarrow X_i}$.

As before, the formal functor $f_+$ preserves quasi-coherence and, if $f$ is proper, preserves coherence.

Finaly, let $\calX$, $\calY$ and $\calZ$ be three formal schemes over $\calS$, $g:\calY\rightarrow\calZ$ and $f=id_\calX\times g:\calX\times\calY\rightarrow\calX\times\calZ$ and assume that $\calE^\cdot\in D^b_{qcoh}(\Dhz_{\calX\times\calY/\calY})$. One can define the relative direct image of $\calE^\cdot$ using the same idea as before: replacing all the $/S$ in the definition by $/Y$ or $/Z$.

\begin{defi}

	For all $i$, the transfer bimodule (of level 0) relative to $Y_i$ and $Z_i$ is defined by
	\[\mathcal{D}^{(0)(Y_i,Z_i)}_{X_i\times Z_i\leftarrow X_i\times Y_i/S_i}=\omega_{X_i\times Y_i/Y_i}\otimes_{\mathcal{O}_{X_i\times Y_i}}\mathcal{D}^{(0)(Y_i,Z_i)}_{X_i\times Y_i\rightarrow X_i\times Z_i}\otimes_{f^{-1}\mathcal{O}_{X_i\times Z_i}}f^{-1}\omega^{-1}_{X_i\times Z_i/Z_i}.\]
	
	The transfer bimodule (of level 0) relative to $\mathcal{Y}$ and $\mathcal{Z}$ is then defined as
	\[\Dh^{(0)(\mathcal{Y},\mathcal{Z})}_{\mathcal{X}\times\mathcal{Z}\leftarrow\mathcal{X}\times\mathcal{Y}/\calS}=\underset{i}{\varprojlim}\mathcal{D}^{(0)(Y_i,Z_i)}_{X_i\times Z_i\leftarrow X_i\times Y_i}.\]
	
	For all $\mathcal{E}_i^\cdot\in D^b(\Dz_{X_i\times Y_i/Y_i})$, the direct image relative to $Y_i$ and $Z_i$ of $\mathcal{E}^\cdot_i$ by $f_i$ is defined as
	\[f_{i+/S_i}^{(Y_i,Z_i)}\mathcal{E}_i^\cdot=Rf_{i*}\left(\mathcal{D}^{(0)(Y_i,Z_i)}_{X_i\times Z_i\leftarrow X_i\times Y_i}\overset{\mathbb{L}}{\otimes}_{\Dz_{X_i\times Y_i/Y_i}}\mathcal{E}^\cdot\right)\in D(\Dz_{X_i\times Z_i/Z_i}).\]
	
	Finally, for all $\mathcal{E}^\cdot\in D^b_{qcoh}(\Dhz_{\mathcal{X}\times\calY/\mathcal{Y}})$, its direct image relative to $\mathcal{Y}$ and $\mathcal{Z}$ is defined by
	\[f^{(\mathcal{Y},\mathcal{Z})}_{+/\calS}\mathcal{E}^\cdot = R\underset{i}{\varprojlim}\left(f_{i+}^{(Y_i,Z_i)}\mathcal{E}^\cdot_i\right).\]

\end{defi}

Again, the proofs of the following properties are straightforward.

\begin{prop}

\begin{itemize}

\item $f_{i+}^{(Y_i,Z_i)}$ and $f^{(\mathcal{Y},\mathcal{Z})}_+$ verify the transitivity formulas.

\item As $\calO_{X_i\times Z_i}$-modules, $f^{(Y_i,Z_i)}_{i+}\mathcal{E}_i^\cdot\simeq Rf_{i*}\mathcal{E}_i^\cdot$.

\item As $\calO_{\calX\times\calZ}$-modules, $f^{(\mathcal{Y},\mathcal{Z})}_+\mathcal{E}^\cdot\simeq Rf_{*}\mathcal{E}^\cdot$.

\item If $\mathcal{E}_i^\cdot\in D^b_{qcoh}(\Dz_{X_{Y,i}/Y_i})$, then $f_{i+}^{(Y_i,Z_i)}\mathcal{E}_i^\cdot\in D^b_{qcoh}(\Dz_{X_{Z,i}/Z_i})$.

\item If $\mathcal{E}^\cdot\in D^b_{qcoh}(\Dhz_{\mathcal{X}_\mathcal{Y}/\mathcal{Y}})$, then $f_+^{(\mathcal{Y},\mathcal{Z})}\mathcal{E}^\cdot\in D^b_{qcoh}(\Dhz_{\mathcal{X}_\mathcal{Z}/\mathcal{Z}})$.

\item If $\mathcal{E}^\cdot\in D^b_{coh}(\Dhz_{\mathcal{X}_\mathcal{Y}/\mathcal{Y}})$ there is a canonical isomorphism

\[Rf_*\left(\Dh^{(0)(\mathcal{Y},\mathcal{Z})}_{\mathcal{X}\times\mathcal{Z}\leftarrow\mathcal{X}\times\mathcal{Y}}\overset{\mathbb{L}}{\otimes}_{\Dhz_{\mathcal{X}\times\mathcal{Y}/\mathcal{Y}}}\mathcal{E}^\cdot\right)\simeq f_+^{(\mathcal{Y},\mathcal{Z})}\mathcal{E}^\cdot.\]

\item If $f_i$ (resp. $f$) is proper, then $f_{i+}^{(Y_i,Z_i)}$ (resp. $f_+^{!(\mathcal{Y},\mathcal{Z})}$) preserves coherence.

\end{itemize}

\end{prop}

\subsection{Fundamental properties}

The goal of this section is to give some important 'calculus' results on relative $\Dhz$-modules, such as the projection formula or the base change formula. It is important to note that these results are isomorphisms whose morphisms are constructed using adjunction formulas, hence it is sufficient to prove the adjunction formulas for relative $\Dhz$-modules, as it is sufficient for a morphism of $\Dhz$-modules to be an isomorphism on the underlying $\calO$-modules in order to be an isomorphism of $\Dhz$-modules.

First of all, recall the adjunction formulas for non-relative $\Dhz$-modules. See \cite{Virrion} sections IV.4 et IV.5 for more details.

\begin{prop}[Adjunction formulas]

Let $f:\mathcal{X}\rightarrow\mathcal{Y}$ be a proper morphism of smooth formal schemes, $\mathcal{E}^\cdot\in D^b_{coh}(\Dhz_\mathcal{X})$ and $\mathcal{F}^\cdot\in D^b_{qcoh}(\Dhz_\mathcal{Y})$.
Then for all $i$ \[Rf_{i*}R\mathcal{H}om_{\Dz_{X_i}}(\mathcal{E}_i^\cdot,f_i^!\mathcal{F}_i^\cdot)\simeq
R\mathcal{H}om_{\Dz_{Y_i}}(f_{i+}\mathcal{E}_i^\cdot,\mathcal{F}_i^\cdot)\]
and \[Rf_{*}R\mathcal{H}om_{\Dhz_\mathcal{X}}(\mathcal{E}^\cdot,f^!\mathcal{F}^\cdot)\simeq
R\mathcal{H}om_{\Dhz_\mathcal{Y}}(f_{+}\mathcal{E}^\cdot,\mathcal{F}^\cdot).\]

In particular, there are morphisms
\[\mathcal{E}^\cdot_i\rightarrow f_i^! f_{i+}\mathcal{E}^\cdot_i \text{ and } \mathcal{E}^\cdot\rightarrow f^! f_{+}\mathcal{E}^\cdot\]
and if $f^!\mathcal{F}^\cdot\in D^b_{coh}(\Dhz_\mathcal{Y})$, there are morphisms
\[f_{i+} f^!_i\mathcal{F}^\cdot_i\rightarrow \mathcal{F}^\cdot_i \text{ and } f_{+} f^!\mathcal{F}^\cdot\rightarrow \mathcal{F}^\cdot.\]

\end{prop}

There are also adjunction formulas for relative $\Dhz$-modules.

\begin{prop}[Relative adjunction morphisms]\label{rel_adj}
Let $\mathcal{X}$, $\mathcal{Y}$ and $\mathcal{Z}$ be three formal schemes, $f:\mathcal{X}_\mathcal{Y}=\mathcal{X}\times \mathcal{Y}\rightarrow \mathcal{X}_\mathcal{Z}=\mathcal{X}\times \mathcal{Z}$ a morphism of the form $id_\mathcal{X}\times g$, $\mathcal{E}^\cdot\in D^b_{qcoh}(\Dhz_{\mathcal{X}_\mathcal{Y}/\mathcal{Y}})$ and $\mathcal{F}^\cdot\in D^b_{qcoh}(\mathcal{D}_{\mathcal{X}_\mathcal{Z}/\mathcal{Z}})$. The adjunction morphisms
\[f_i^{!(Y_i,Z_i)}f_{i+}^{(Y_i,Z_i)}\mathcal{E}_i^\cdot\rightarrow\mathcal{E}_i^\cdot\text{ et } \mathcal{F}_i^\cdot\rightarrow f_{i+}^{(Y_i,Z_i)}f_i^{!(Y_i,Z_i)}\mathcal{F}_i^\cdot\]
are respectively morphisms on the categories $D^b_{qcoh}(\Dz_{X_{Y,i}/Y_i})$ and $D^b_{qcoh}(\Dz_{X_{Z,i}/Z_i})$.

There are also adjunction morphisms
\[f^{!(\mathcal{Y},\mathcal{Z})}f_{+}^{(\mathcal{Y},\mathcal{Z})}\mathcal{E}^\cdot\rightarrow\mathcal{E}^\cdot\text{ et } \mathcal{F}^\cdot\rightarrow f_{+}^{(\mathcal{Y},\mathcal{Z})}f^{!(\mathcal{Y},\mathcal{Z})}\mathcal{F}^\cdot\]
in the categories $D^b_{qcoh}(\Dhz_{\mathcal{X}_\mathcal{Y}/\mathcal{Y}})$ and $D^b_{qcoh}(\Dhz_{\mathcal{X}_\mathcal{Z}/\mathcal{Z}})$.
\end{prop}

\begin{proof}

As the following reasoning will work on all the morphisms $f_i:X_i\times Y_i\rightarrow X_i\times Z_i$, we will simplify the notations and consider a morphism $f:X\times Y\rightarrow X\times Z$.

It is known (\cite[\href{https://stacks.math.columbia.edu/tag/0096}{Lemme 0096}]{Sta_proj}) that given $\calA$ a sheaf of ring on $X\times Y$, $\calB$ a sheaf of ring on $X\times Z$ and a structure $\calB\rightarrow f_*\calA$ then the functor $\calA\otimes_{f^{-1}\calB}f^{-1}:Mod(\calB)\rightarrow Mod(\calA)$ is left adjoint to the functor $f_*:Mod(\calA)\rightarrow Mod(\calB)$.

Applying this result with $\calA=\Dz_{X\times Y/Y}$ and $\calB=\Dz_{X\times Z/Z}$ and remarking that $\calD^{(Y,Z)}_{X\times Z\leftarrow X\times Y}\simeq\Dz_{X\times Y/Y}$ (modulo the inversion of the left and right structure), one can say that the functors
\[f_*(\calD^{(Y,Z)}_{X\times Z\leftarrow X\times Y}\otimes_{\Dz_{X\times Y/Y}}\bullet):Mod(\Dz_{X\times Y/Y})\rightarrow Mod(\Dz_{X\times Z/Z})\]
and
\[\Dz_{X\times Y/Y}\otimes_{f^{-1}\Dz_{X\times Z/Z}}f^{-1}:Mod(\Dz_{X\times Z/Z})\rightarrow Mod(\Dz_{X\times Y/Y})\]
are adjoints. As the first one is left exact ($\calD^{(Y,Z)}_{X\times Z\leftarrow X\times Y}$ is $\Dz_{X\times Y/Y}$-flat) and the second one is right exact, these functors have derived functors, which are still adjoints. These derived functors are $f^{(Y,Z)}_+$ and $f^{!(Y,Z)}$ respectively.

So, for all $i$, we have proved the adjunction between $f_i^{!(Y_i,Z_i)}$ and $f_{i+}^{(Y_i,Z_i)}$. For the one between $f^{!(\mathcal{Y},\mathcal{Z})}$ and $f_{+}^{(\mathcal{Y},\mathcal{Z})}$, it is sufficient to apply the inverse limit on the morphisms modulo $\pi^i$, as these morphisms are compatible to the reductions modulo $\pi^j$.

\end{proof}

\begin{rem}
The adjunction formulas for relative $\Dhz$-modules are not the same as the ones for non-relative $\Dhz$-modules, however, they are the same as the ones for $\calO$-modules. Hence, the following proofs for the results on relative $\Dhz$-modules will be closer to the ones of the results for $\calO$-modules than the ones of the results for non-relative $\Dhz$-modules.
\end{rem}

The projection formula will be an important tool to compute the composition of two Fourier-Mukai transforms. The following statement can be find in \cite{Virrion}, section II.4.

\begin{prop}[Projection formula]\label{proj}

Let $f:\mathcal{X}\rightarrow \mathcal{Y}$ be a proper morphism of smooth formal schemes, $\mathcal{E}^\cdot\in D^b_{qcoh}(\Dhz_{\mathcal{X}/\mathcal{S}})$ and $\mathcal{F}^\cdot\in D^b_{qcoh}(\Dhz_{\mathcal{Y}/\mathcal{S}})$.

Then for all $i$ there is a natural isomorphism
\[f_{i+}(f_i^\flat\mathcal{F}_i^\cdot\overset{\mathbb{L}}{\otimes}_{\mathcal{O}_{X_i}}\mathcal{E}_i^\cdot) \simeq \mathcal{F}_i^\cdot \overset{\mathbb{L}}{\otimes}_{\mathcal{O}_{Y_i}}f_{i+}\mathcal{E}_i^\cdot.\]
(Where the functor $f_i^\flat$ represents the inverse image without the shift.)

This isomorphism is compatible to the reductions modulo $\pi^j$, such that it naturally induces an isomorphism
\[f_+(f^\flat\mathcal{F}^\cdot\hat{\otimes}^{\mathbb{L}}_{\mathcal{O}_{\mathcal{X}}}\mathcal{E}^\cdot) \simeq \mathcal{F}^\cdot \hat{\otimes}^\mathbb{L}_{\mathcal{O}_{\mathcal{Y}}}f_+\mathcal{E}^\cdot.\]
%(Comme dans le chapitre précédent, le foncteur $p^\flat$ représente l'image inverse sans décalage.)

\end{prop}

The same result is true for relative $\Dhz$-modules.

\begin{prop}[Relative projection formula]\label{proj_rel}

Let $f:\mathcal{X}_\mathcal{Y} \rightarrow \mathcal{X}_\mathcal{Z}$ be a morphism of the form $id_\mathcal{X}\times g$, $\mathcal{E}^\cdot\in D^b_{qcoh}(\Dhz_{\mathcal{X}_\mathcal{Z}/\mathcal{Z}})$ and $\mathcal{F}^\cdot\in D^b_{qcoh}(\Dhz_{\mathcal{X}_\mathcal{Y}/\mathcal{Y}})$. Then
\[f^{(Y_i,Z_i)}_{i+}(f_i^{!(Y_i,Z_i)}\mathcal{E}_i^\cdot\overset{\mathbb{L}}{\otimes}_{\mathcal{O}_{X_{Y,i}}}\mathcal{F}_i^\cdot) \simeq \mathcal{E}_i^\cdot \overset{\mathbb{L}}{\otimes}_{\mathcal{O}_{X_{Z,i}}}f^{(Y_i,Z_i)}_{i+}\mathcal{F}_i^\cdot.\]
Moreover,
\[f^{(\mathcal{Y},\mathcal{Z})}_+(f^{!(\mathcal{Y},\mathcal{Z})}\mathcal{E}^\cdot\hat{\otimes}^{\mathbb{L}}_{\mathcal{O}_{\mathcal{X}_\mathcal{Y}}}\mathcal{F}^\cdot) \simeq \mathcal{E}^\cdot \hat{\otimes}^\mathbb{L}_{\mathcal{O}_{\mathcal{X}_\mathcal{Z}}}f^{(\mathcal{Y},\mathcal{Z})}_+\mathcal{F}^\cdot.\]

\end{prop}

\begin{proof}

As explained earlier, as the adjunction formulas are the same for $\calO$-modules and relative $\Dh$-modules over the $X_i\times Y_i$ and $X_i\times Z_i$, there exists a morphism of $\Dh_{X_i\times Z_i/Z_i}$-modules
\[\mathcal{E}_i^\cdot \overset{\mathbb{L}}{\otimes}_{\mathcal{O}_{X_{Z,i}}}f^{(Y_i,Z_i)}_{i+}\mathcal{F}_i^\cdot\rightarrow f^{(Y_i,Z_i)}_{i+}(f_i^{!(Y_i,Z_i)}\mathcal{E}_i^\cdot\overset{\mathbb{L}}{\otimes}_{\mathcal{O}_{X_{Y,i}}}\mathcal{F}_i^\cdot).\]
To prove that it is an isomorphism, it is sufficient to consider the underlying $\calO$-modules, which are isomorphic as settled by the projection formula for $\calO$-modules.

For the second isomorphism, as adjunction formulas are compatible with reductions modulo $\pi^i$, so are the first isomorphims and one can take their inverse limit.

\end{proof}

For the base change formula, it works the same way. First recall the non-relative result.

\begin{prop}[Base change]\label{chgmt_base}
Let $\mathcal{X}$, $\mathcal{Y}$ and $\mathcal{Z}$ be three smooth formal schemes, $f:\mathcal{X}\rightarrow \mathcal{Y}$ and $\pi_\mathcal{X}:\mathcal{X}\times\mathcal{Z}\rightarrow \mathcal{X}$ and $\pi_\mathcal{Y}:\mathcal{Y}\times\mathcal{Z}\rightarrow \mathcal{Y}$ the projections. Consider the cartesian diagram
\[\xymatrix@=40pt{
	\mathcal{X}\times\mathcal{Z} \ar^{\pi_\mathcal{Y}}[r] \ar_{f\times id_\mathcal{Z}}[d] & \mathcal{X} \ar^f[d] \\
	\mathcal{Y}\times\mathcal{Z} \ar_{\pi_\mathcal{X}}[r] & \mathcal{Y}
}\]
There is an isomorphism of functors $\pi_\mathcal{Y}^!f_+\simeq(f\times id_\mathcal{Z})_+\pi_\mathcal{X}^!$.

\noindent There is also an isomorphism of relative functors $\pi_\mathcal{Y}^{!(\mathcal{Z})}f_+\simeq(f\times id_\mathcal{Z})_{+/\mathcal{Z}}\pi_\mathcal{X}^{!(\mathcal{Z})}$.
\end{prop}

It is important to note that, as we are working on positive characteristic, the usual proof of this result (see \cite{hotta} e.g.) can't be applied directly as it uses Kashiwara's theorem, which is false in positive characteristic. For that reason, one has to consider only projections for the horizontal arrows, as the usual proof only uses Kashiwara's theorem for closed immersions.

%In the relative case, however, there is no need to restrict some morphisms to be projections (the horizontal arrows still have to be proper though), as the result is proved the same way as the one for $\calO$-modules. However, for the sake of clarity and as this result will only be used in that case, the relative base change will be written with all four arrows being projections.

\begin{prop}[Relative base change]\label{chgmt_base_rel}
Let $\mathcal{X}$, $\mathcal{Y}_1$, $\mathcal{Y}_2$ and $\mathcal{Z}$ be four smooth formal schemes, $\pi_1:\mathcal{X}\times\mathcal{Y}_1\times\mathcal{Z}\rightarrow \mathcal{X}\times\mathcal{Y}_1$ et $\pi_2:\mathcal{X}\times\mathcal{Y}_2\times\mathcal{Z}\rightarrow \mathcal{X}\times\mathcal{Y}_2$ the projections and $f:\mathcal{Y}_1\rightarrow \mathcal{Y}_2$.
Consider $g=id_\mathcal{X}\times f$, $h=id_\mathcal{X}\times f\times id_\mathcal{Z}$ and the following cartesian diagram
\[\xymatrix@=40pt{
	\mathcal{X}\times\mathcal{Y}_1\times\mathcal{Z} \ar^{\pi_1}[r] \ar_{h}[d] & \mathcal{X}\times\mathcal{Y}_1 \ar^{g}[d] \\
	\mathcal{X}\times\mathcal{Y}_2\times\mathcal{Z} \ar_{\pi_2}[r] & \mathcal{X}\times\mathcal{Y}_2
}\]
Then there is an isomorphism of functors $\pi^!_{2/\mathcal{Y}_2}g^{(\mathcal{Y}_1,\mathcal{Y}_2)}_+\simeq h^{(\mathcal{Y}_1,\mathcal{Y}_2)}_+\pi^!_{1/\mathcal{Y}_1}$.

\noindent There is also an isomorphism of functors $\pi^{!(\mathcal{Z})}_{2/\mathcal{Y}_2}g^{(\mathcal{Y}_1,\mathcal{Y}_2)}_+\simeq h^{(\mathcal{Y}_1,\mathcal{Y}_2)}_+\pi^{!(\mathcal{Z})}_{1/\mathcal{Y}_1}$.

\end{prop}

The idea of the proof is the same as the one for the relative projection formula.

Note that, as Kashiwara's theorem isn't needed here, one can write this result for horizontal arrows being proper and not necessary projections. However, as we won't need more generality in the following this result is enough.

Another important "calculus rule" is that the relative functors are compatible with the non-relative ones. For the sake of clarity the following morphisms will be supposed to be projections, but these results are true for all morphisms as they are nothing more than elementary calculations.

\begin{prop}\label{compat}
Consider the following projections

	\[\xymatrix@=30pt{
			\mathcal{X}\times\mathcal{Y}\times\mathcal{Z} \ar^{\pi_\mathcal{YZ}}[r] \ar_{\pi_\mathcal{XZ}}[d] & \mathcal{Y}\times\mathcal{Z} \ar^{f}[d] \\
			\mathcal{X}\times \mathcal{Z} \ar_{g}[r] & \mathcal{Z}
	}\]
	
Consider $\mathcal{E}^\cdot\in D^b_{qcoh}(\mathcal{D}_{\mathcal{Z}/\mathcal{S}})$ and $\mathcal{F}^\cdot\in D^b_{qcoh}(\mathcal{D}_{\mathcal{X}\times\mathcal{Y}\times\mathcal{Z}/\mathcal{Y}})$. The following isomorphims hold:

\[\pi_{X_iZ_i/S_i}^{!(Y_i)}\circ g_{i/S_i}^\flat(\mathcal{E}_i^\cdot) \simeq\pi_{Y_iZ_i/Y_i}^\flat\circ f_{i/S_i}^{!(Y_i)}(\mathcal{E}_i^\cdot) \text{ } (\text{in the category }D^b_{qcoh}(\Dz_{X_i\times Y_i\times Z_i/Y_i})),\]
\[g_{i+/S_i}\circ\pi_{X_iZ_i+/S_i}^{(Y_i)}(\mathcal{F}_i^\cdot) \simeq f_{i+/S_i}^{(Y_i)}\circ\pi_{Y_iZ_i+/Y_i}(\mathcal{F}_i^\cdot)\text{ } (\text{in the category }D^b_{qcoh}(\Dz_{Z_i/S_i})),\]

\[\pi_{\mathcal{XZ}/\mathcal{S}}^{!(\mathcal{Y})}\circ g_{/\mathcal{S}}^\flat(\mathcal{E}^\cdot) \simeq\pi_{\mathcal{YZ}/\mathcal{Y}}^\flat\circ f_{/\mathcal{S}}^{!(\mathcal{Y})}(\mathcal{E}^\cdot)\text{ } (\text{in the category }D^b_{qcoh}(\Dhz_{\mathcal{X}\times\mathcal{Y}\times\mathcal{Z}/\mathcal{Y}})),\]
\[g_{+/\mathcal{S}}\circ\pi_{\mathcal{XZ}+/\mathcal{S}}^{(\mathcal{Y})}(\mathcal{F}^\cdot) \simeq f_{+/\mathcal{S}}^{(\mathcal{Y})}\circ\pi_{\mathcal{YZ}+/\mathcal{Y}}(\mathcal{F}^\cdot)\text{ } (\text{in the category }D^b_{qcoh}(\Dhz_{\mathcal{Z}/\mathcal{S}})).\]
\end{prop}

The isomorphisms in this section may be troublesome to understand because of the large amount of subscripts and superscript, so keep in mind one thing: if you are writing functors and objects that make sense (especially in term of what are the bases whose sheaves are considered relative to), all the expected isomorphisms will hold. The only thing to be carefull of is that the horizontal arrows for the non-relative base change have to be projections, but as in the following we will work only with projections, there will be no problem.

\subsection{Seesaw principle}

One interesting application of relative $\Dhz$-modules is that one can find a generalization of the seesaw principle (see \cite{Mum} corollary 5.6 for the usual version). Even though we won't use it in the rest of this article, it is an interesting result on its own.

\begin{prop}[Seesaw principle]\label{D-bascule}
Let $\mathcal{X}$, $\mathcal{Y}$ and $\mathcal{Z}$ be three formal schemes, $\mathcal{X}$ and $\mathcal{Y}$ complete and $\mathcal{L}$ a sheaf of $\Dhz_{\mathcal{X}\times\mathcal{Y}\times\mathcal{Z}/\mathcal{Y}\times\mathcal{Z}}$-modules $\mathcal{O}_{\mathcal{X}\times\mathcal{Y}\times\mathcal{Z}}$-invertible.

Consider $\pi_\mathcal{XY}:\mathcal{X}\times\mathcal{Y}\times\mathcal{Z}\rightarrow \mathcal{X}\times\mathcal{Y}$ the projection, for all $(x,y)\in\mathcal{X}\times\mathcal{Y}$, 
\[\iota_{(x,y)}:\mathcal{Z}\simeq \{x\}\times\{y\}\times \mathcal{Z}\hookrightarrow\mathcal{X}\times\mathcal{Y}\times\mathcal{Z}\]
and for all $z\in \mathcal{Z}$, 
\[\tilde{\iota}_z:\mathcal{X}\times\mathcal{Y}\simeq \mathcal{X}\times\mathcal{Y}\times\{z\}\hookrightarrow \mathcal{X}\times\mathcal{Y}\times\mathcal{Z}.\]
\begin{itemize}
\item If $\forall (x,y)\in \mathcal{X}\times\mathcal{Y}$, $\iota_{(x,y)}^*\mathcal{L}\simeq\mathcal{O}_\mathcal{Z}$ (as $\mathcal{O}_\mathcal{Z}$-modules),
then there exists a $\Dhz_{\mathcal{X}\times\mathcal{Y}/\mathcal{Y}}$-module $\mathcal{O}_{\mathcal{X}\times\mathcal{Y}}$-invertible $\mathcal{M}$ such that $\mathcal{L}\simeq\pi_{\mathcal{XY}/\mathcal{Y}}^{!(\mathcal{Z})}\mathcal{M}$ (as $\Dhz_{\mathcal{X}\times\mathcal{Y}\times\mathcal{Z}/\mathcal{Y}\times\mathcal{Z}}$-modules).
\item If moreover there exists $z\in \mathcal{Z}$ such that $\tilde{\iota}_{z/\mathcal{Y}}^{!(\mathcal{Z})}\mathcal{L}\simeq\mathcal{O}_{\mathcal{X}\times\mathcal{Y}}$ (as $\Dhz_{\mathcal{X}\times\mathcal{Y}/\mathcal{Y}}$-modules),
then $\mathcal{L}\simeq\mathcal{O}_{\mathcal{X}\times\mathcal{Y}\times\mathcal{Z}}$ (as $\Dhz_{\mathcal{X}\times\mathcal{Y}\times\mathcal{Z}/\mathcal{Y}\times\mathcal{Z}}$-modules).
\end{itemize}
\end{prop}

%Note that this result is also true if $X$, $Y$ and $Z$ are classical schemes.

\begin{proof}

In order to prove this $\calD$-seesaw-principle, we will first prove it in the case when $X$, $Y$ and $Z$ are non-formal schemes. The main idea is to adjust the proof of \cite{Mum} in order to take into account the structure of $\calD$-module of the sheaves.

The first hypothesis allows to apply the seesaw principle for $\calO$-modules: there exists an invertible $\calO_{X\times Y}$-module $\calM$ such that $\calL\simeq \pi^*_{XY}\calM$. Looking in details to the proof of \cite{Mum} corollary 5.6, one sees that $\calM$ can be taken as $\pi_{XY*}\calL$. The idea is then to find an appropriate structure of $\calD_{X\times Y/Y}$-module on $\pi_{XY*}\calL$. Consider
\[\calM=\pi_{XY*}(\calD^{(Y,Y\times Z)}_{X\times Y\leftarrow X\times Y\times Z}\otimes_{\calD_{X\times Y\times Z/Y\times Z}}\calL).\]
It is important to note that none of the functors are derivated, hence $\calM$ is a sheaf (not a complex). Its structure of $\calD_{X\times Y/Y}$-module is defined by the structure of $\pi_{XY}^{-1}\calD_{X\times Y\times Z/Y\times Z}$-module of $\calD^{(Y,Y\times Z)}_{X\times Y\leftarrow X\times Y\times Z}$ (the same way as it is for $\pi_{XY+/Y}^{(Z)}\calL$). Finally, as $\calO_{X\times Y}$-modules, $\calM\simeq\pi_{XY*}\calL$, so the seesaw principle for $\calO$-modules sates that $\calM$ is $\calO_{X\times Y}$-invertible and, as $\calO_{X\times Y\times Z}$-modules, $\pi_{XY}^*\calM\simeq\calL$.

However, there is an adjunction morphism $\pi_{XY/Y}^{!(Z)}\pi_{XY*}(\calD^{(Y,Y\times Z)}_{X\times Y\leftarrow X\times Y\times Z}\otimes_{\calD_{X\times Y\times Z/Y\times Z}}\calL)\rightarrow\calL$ in the category of $\calD_{X\times Y\times Z/Y\times Z}$-modules (it is the non-derivated one used in the proof of property \ref{rel_adj}, as $\calD_{X\times Y\times Z/Y\times Z}\otimes_{\pi_{XY}^{-1}\calD_{X\times Y/Y}}\pi_{XY}^{-1}$ is exact because $\pi_{XY}$ is a projection). So there is a morphism of $\calD_{X\times Y\times Z/Y\times Z}$-modules $\pi_{XY/Y}^{!(Z)}\calM\rightarrow\calL$ which is an isomorphism of $\calO_{X\times Y\times Z}$-modules, so it is an isomorphism of $\calD_{X\times Y\times Z/Y\times Z}$-modules.

If moreover there exists $z\in Z$ such that $\tilde{\iota}_{z/Y}^{!(Z)}\mathcal{L}\simeq\mathcal{O}_{X\times Y}$ as $\calD_{X\times{Y}/Y}$-modules, then as $\pi_{XY}\circ\tilde{\iota}_z=id_{Y\times Z}$
\[\calM\simeq\tilde{\iota}_{z/{Y}}^{!({Z})}\pi^{!(Z)}_{XY/Y}\calM\simeq\tilde{\iota}_{z/{Y}}^{!({Z})}\calL\simeq\calO_{X\times Y},\]
as $\calD_{X\times Y/Y}$-modules. Hence, $\calL\simeq\calO_{X\times Y\times Z}$ as $\calD_{X\times Y\times Z/Y\times Z}$-modules.

The seesaw principle is then proved for all the $X_i$, $Y_i$ and $Z_i$, reductions modulo $\pi^i$ of $\calX$, $\calY$ and $\calZ$. In order to prove the seesaw principle for formal schemes, it only necessary to remark that the isomorphisms expected are obtained via adjunction formulas, and hence are compatible with reductions modulo $\pi^j$. Considering the $\Dhz_{\calX\times\calY/\calY}$-module
\[\calM=\varprojlim\calM_i=\varprojlim \pi_{X_iY_i*}(\calD^{(Y_i,Y_i\times Z_i)}_{X_i\times Y_i\leftarrow X_i\times Y_i\times Z_i}\otimes_{\Dz_{X_i\times Y_i\times Z_i/Y_i\times Z_i}}\calL_i),\]
one have the expected isomorphism of $\Dhz_{\calX\times\calY\times\calZ/\calY\times\calZ}$-modules $\calL\simeq\pi_{\calX\calY/\calY}^{!(\calZ)}\calM$. Moreover, as the $\calM_i$ are $\calO$-invertible, so is $\calM$.

The second part of the seesaw principle is proved the same way as in the non-formal case.

\end{proof}

\begin{rem}
This result is written in its more general case but can be greatly simplified by considering $\calY=\calS$. Moreover, considering $\calX=\calS$ one can regain the seesaw principle for $\calO$-modules.
\end{rem}

\section{Fourier-Mukai transform for $\Dhz$-modules}

In this section, we will first construct $\Dhz$-Fourier-Mukai transforms of a given kernel on formal schemes (non necessary abelian) then consider the $\Dhz$-Fourier-Mukai transform of Poincar\'e kernel on formal abelian schemes.

As in the previous section, let $V$ be a discrete valuation ring with uniformizer $\pi$, $S=Spec(V)$, for all $i$ $V_i=\faktor{V}{\pi^i V}$ and $S_i=Spec{V_i}$ and $\calS=Spf{S}$. The non-necessary-abelian formal schemes will be denoted $\calX$, $\calY$ and $\calZ$ and the formal abelian scheme will be denoted $\calA$.

\subsection{Fourier-Mukai transforms of a given kernel}

This section is an adaptation of what is done for $\calO$-modules over classical varieties in \cite{Huy}. As all the results needed have been previously settled, this section will be close to the section 5.1.

Let $\calX$, $\calY$ and $\calZ$ be three quasi-compact smooth formal schemes.

Given $\calX$, $\calY$ and a kernel $\calP\in D^b_{qcoh}(\calO_{\calX\times\calY})$ one can define up to four different Fourier-Mukai transforms depending on the additional structure of $\Dhz$-module of $\calP$ we consider.

\begin{defi}

Let $p:\calX\times\calY\rightarrow\calX$ and $q:\calX\times\calY\rightarrow\calY$ be the projections, $\calP\in D^b_{qcoh}(\calO_{\calX\times\calY})$, $\calQ\in D^b_{qcoh}(\Dhz_{\calX\times\calY/\calS})$ and $\calR\in D^b_{qcoh}(\Dhz_{\calX\times\calY/\calY})$.

\begin{itemize}

\item The Fourier-Mukai transform of kernel $\calP$ is the functor $\Phi_\mathcal{P} : D^b_{qcoh}(\mathcal{O}_\mathcal{X})\rightarrow D^b_{qcoh}(\mathcal{O}_\mathcal{Y})$ defined as
\[\Phi_\mathcal{P}(\mathcal{E}^\cdot)=Rq_*(\mathcal{P}\hat{\otimes}^{\mathbb{L}}_{\mathcal{O}_{\mathcal{X}\times\mathcal{Y}}}p^*\mathcal{E}^\cdot),\]
where the functor $p^*$ is defined as $p^*\mathcal{E}^\cdot=R\varprojlim p_i^*\mathcal{E}_i^\cdot$.

\item The Fourier-Mukai transform from $\calX$ to $\calY$ of kernel $\calQ$ is the functor $\Phi^{\Dhz}_\mathcal{Q} : D^b_{qcoh}(\Dhz_{\mathcal{X/S}})\rightarrow D^b_{qcoh}(\Dhz_\mathcal{Y/S})$ defined as
\[\Phi^{\Dhz}_\mathcal{Q}(\mathcal{E}^\cdot)=q_+(\mathcal{Q}\hat{\otimes}^{\mathbb{L}}_{\mathcal{O}_{\mathcal{X}\times\mathcal{Y}}}p^\flat\mathcal{E}^\cdot),\]
where the functor $p^\flat$ is the inverse image without the shift: $p^!\calE=p^\flat\calE[dim(\calX)-dim(\calY)]$.

\item The Fourier-Mukai transform from $\calX$ to $\calY$ of kernel $\calR$ (relative to $\calY$) is the functor $\Phi^{\mathcal{X}\rightarrow \mathcal{Y}}_\mathcal{R} : D^b_{qcoh}(\mathcal{D}_\mathcal{X/S})\rightarrow D^b_{qcoh}(\mathcal{O}_\mathcal{Y})$ defined as
\[\Phi^{\mathcal{X}\rightarrow \mathcal{Y}}_\mathcal{R}(\mathcal{E}^\cdot)=q_{+/\mathcal{Y}}(\mathcal{R}\hat{\otimes}^{\mathbb{L}}_{\mathcal{O}_{\mathcal{X}\times \mathcal{Y}}}p^{!(\mathcal{Y})}\mathcal{E}^\cdot).\]

\item The Fourier-Mukai transform from $\calY$ to $\calX$ of kernel $\calR$ (relative to $\calY$) is the functor $\Phi^{\mathcal{X}\leftarrow \mathcal{Y}}_\mathcal{R} : D^b_{qcoh}(\mathcal{O}_\mathcal{Y})\rightarrow D^b_{qcoh}(\mathcal{D}_\mathcal{X/S})$ defined as
\[\Phi^{\mathcal{X}\leftarrow\mathcal{Y}}_\mathcal{R}(\mathcal{E}^\cdot)=p_+^{(\mathcal{Y})}(\mathcal{Q}\hat{\otimes}^{\mathbb{L}}_{\mathcal{O}_{\mathcal{X}\times\mathcal{Y}}}q^\flat_{/\calY}\mathcal{E}^\cdot).\]

\end{itemize}

\end{defi}

As expected, the composition of two Fourier-Mukai transforms is still a Fourier-Mukai transform. However, given $\calX$, $\calY$ and $\calZ$ three formal schemes, there are eight possible compositions of two Fourier-Mukai transforms, depending on which formal scheme the derivations are relative to.

\begin{prop}

Let $\calX$, $\calY$ and $\calZ$ be three smooth quasi-compact formal schemes and $\pi_{\calX\calY}$, $\pi_{\calX\calZ}$ and $\pi_{\calY\calZ}$ the projections from $\calX\times\calY\times\calZ$ to $\calX\times\calY$, $\calX\times\calZ$ and $\calY\times\calZ$ respectively.

\begin{itemize}

\item Given $\calP\in D^b_{qcoh}(\calO_{\calX\times\calY})$ and $\calQ\in D^b_{qcoh}(\calO_{\calY\times\calZ})$, denote
\[\calR=\pi_{\calX\calZ*}(\pi^*_{\calX\calY}\calP\hat{\otimes}^\mathbb{L}_{\calO_{\calX\times\calY\times\calZ}}\pi^*_{\calY\calZ}\calQ)\in D^b_{qcoh}(\calO_{\calX\times\calZ}).\]
Then $\Phi_\calQ\circ\Phi_\calP\simeq\Phi_\calR$.

\item If $\calX$ is proper, given $\calP\in D^b_{qcoh}(\Dhz_{\calX\times\calY/\calY})$ and $\calQ\in D^b_{qcoh}(\calO_{\calY\times\calZ})$, denote
\[\calR=\pi_{\calX\calZ+/\calZ}^{(\calY)}(\pi^{!(\calZ)}_{\calX\calY/\calY}\calP\hat{\otimes}^\mathbb{L}_{\calO_{\calX\times\calY\times\calZ}}\pi^*_{\calY\calZ}\calQ)\in D^b_{qcoh}(\Dhz_{\calX\times\calZ/\calZ}).\]
Then $\Phi_\calQ\circ\Phi^{\calX\rightarrow\calY}_\calP\simeq\Phi^{\calX\rightarrow\calZ}_\calR$.

\item If $\calY$ is proper, given $\calP\in D^b_{qcoh}(\Dhz_{\calX\times\calY/\calX})$ and $\calQ\in D^b_{qcoh}(\Dhz_{\calY\times\calZ/\calZ})$, denote
\[\calR=\pi_{\calX\calZ+/\calX\times\calZ}(\pi^{!(\calZ)}_{\calX\calY/\calX}\calP\hat{\otimes}^\mathbb{L}_{\calO_{\calX\times\calY\times\calZ}}\pi^{!(\calX)}_{\calY\calZ/\calZ}\calQ)\in D^b_{qcoh}(\calO_{\calX\times\calZ}).\]
Then $\Phi^{\calY\rightarrow\calZ}_\calQ\circ\Phi^{\calY\leftarrow\calX}_\calP\simeq\Phi_\calR$.

\item If $\calZ$ is proper, given $\calP\in D^b_{qcoh}(\calO_{\calX\times\calY})$ and $\calQ\in D^b_{qcoh}(\Dhz_{\calY\times\calZ/\calY})$, denote
\[\calR=\pi_{\calX\calZ+/\calX}^{(\calY)}(\pi^*_{\calX\calY}\calP\hat{\otimes}^\mathbb{L}_{\calO_{\calX\times\calY\times\calZ}}\pi^{!(\calX)}_{\calY\calZ/\calY}\calQ)\in D^b_{qcoh}(\Dhz_{\calX\times\calZ/\calX}).\]
Then $\Phi^{\calZ\leftarrow\calY}_\calQ\circ\Phi_\calP\simeq\Phi^{\calZ\leftarrow\calX}_\calR$.

\item If $\calX$ and $\calY$ are proper, given $\calP\in D^b_{qcoh}(\Dhz_{\calX\times\calY})$ and $\calQ\in D^b_{qcoh}(\Dhz_{\calY\times\calZ/\calZ})$, denote
\[\calR=\pi_{\calX\calZ+/\calZ}(\pi^{!(\calZ)}_{\calX\calY}\calP\hat{\otimes}^\mathbb{L}_{\calO_{\calX\times\calY\times\calZ}}\pi^\flat_{\calY\calZ/\calZ}\calQ)\in D^b_{qcoh}(\Dhz_{\calX\times\calZ/\calZ}).\]
Then $\Phi^{\calY\rightarrow\calZ}_\calQ\circ\Phi^{\Dhz}_\calP\simeq\Phi^{\calX\rightarrow\calZ}_\calR$.

\item If $\calY$ and $\calZ$ are proper, given $\calP\in D^b_{qcoh}(\Dhz_{\calX\times\calY/\calX})$ and $\calQ\in D^b_{qcoh}(\Dhz_{\calY\times\calZ})$, denote
\[\calR=\pi_{\calX\calZ+/\calX}(\pi^\flat_{\calX\calY/\calX}\calP\hat{\otimes}^\mathbb{L}_{\calO_{\calX\times\calY\times\calZ}}\pi^{!(\calX)}_{\calY\calZ}\calQ)\in D^b_{qcoh}(\Dhz_{\calX\times\calZ/\calX}).\]
Then $\Phi^{\Dhz}_\calQ\circ\Phi^{\calY\leftarrow\calX}_\calP\simeq\Phi^{\calZ\leftarrow\calX}_\calR$.

\item If $\calX$ and $\calZ$ are proper, given $\calP\in D^b_{qcoh}(\Dhz_{\calX\times\calY/\calY})$ and $\calQ\in D^b_{qcoh}(\Dhz_{\calY\times\calZ/\calY})$, denote
\[\calR=\pi^{(\calY)}_{\calX\calZ+}(\pi^\flat_{\calX\calY/\calY}\calP\hat{\otimes}^\mathbb{L}_{\calO_{\calX\times\calY\times\calZ}}\pi^\flat_{\calY\calZ/\calY}\calQ)\in D^b_{qcoh}(\Dhz_{\calX\times\calZ}).\]
Then $\Phi^{\calZ\leftarrow\calY}_\calQ\circ\Phi^{\calX\rightarrow\calY}_\calP\simeq\Phi^{\Dhz}_\calR$.

\item If $\calX$, $\calY$ and $\calZ$ are proper, given $\calP\in D^b_{qcoh}(\Dhz_{\calX\times\calY})$ and $\calQ\in D^b_{qcoh}(\Dhz_{\calY\times\calZ})$, denote
\[\calR=\pi_{\calX\calZ+}(\pi^\flat_{\calX\calY}\calP\hat{\otimes}^\mathbb{L}_{\calO_{\calX\times\calY\times\calZ}}\pi^\flat_{\calY\calZ}\calQ)\in D^b_{qcoh}(\Dhz_{\calX\times\calZ}).\]
Then $\Phi^{\Dhz}_\calQ\circ\Phi^{\Dhz}_\calP\simeq\Phi^{\Dhz}_\calR$.

\end{itemize}

\end{prop}

\begin{proof}

The proof is only a succession of base changes and projection formulas and the only difference between the proofs of the eight properties are which version of base change and projection formula is used. Hence, we will only prove the seventh one as an exemple.

Asume then $\calP\in D^b_{qcoh}(\Dhz_{\calX\times\calY/\calY})$, $\calQ\in D^b_{qcoh}(\Dhz_{\calY\times\calZ/\calY})$ and
\[\calR=\pi^{(\calY)}_{\calX\calZ+}(\pi^\flat_{\calX\calY/\calY}\calP\hat{\otimes}^\mathbb{L}_{\calO_{\calX\times\calY\times\calZ}}\pi^\flat_{\calY\calZ/\calY}\calQ)\in D^b_{qcoh}(\Dhz_{\calX\times\calZ}).\]

Using the following notations

\[\xymatrix{
	& & &\mathcal{X}\times\mathcal{Y}\times\mathcal{Z} \ar@/_40pt/_{\pi_\mathcal{X}}[llldd] \ar_{\pi_\mathcal{XY}}[dll] \ar^{\pi_\mathcal{XZ}}[dd] \ar^{\pi_\mathcal{YZ}}[rrd] \ar@/^40pt/^{\pi_\mathcal{Z}}[rrrdd] & & & \\
	& \mathcal{X}\times\mathcal{Y} \ar_q[dl] \ar^p[dr] & & & & \mathcal{Y}\times\mathcal{Z} \ar_u[ld] \ar^t[rd]& \\
	\mathcal{X} \ar@{=}[rrd]& &\mathcal{Y} \ar@/^20pt/@{=}[rr] &\mathcal{X}\times\mathcal{Z} \ar_s[dl] \ar^r[rd]&\mathcal{Y} & &\mathcal{Z} \ar@{=}[lld] \\
	& &\mathcal{X} & &\mathcal{Z}. & &
}\]

one has the following isomorphisms

\begin{align*}
\Phi_\calR^{\Dhz}(\calE^\cdot)&= r_+(\mathcal{R}\hat{\otimes} s^\flat\mathcal{E}^\cdot) & \\
&= r_+\left(\pi_{\mathcal{XZ}+}^{(\mathcal{Y})}(\pi^\flat_\mathcal{XY/Y}\mathcal{P}\hat{\otimes}\pi^\flat_\mathcal{YZ/Y}\mathcal{Q})\hat{\otimes} s^\flat\mathcal{E}^\cdot\right) &  \\
 &\simeq r_+\left( \pi_{\mathcal{XZ}+}^{(\mathcal{Y})}(\pi^\flat_\mathcal{XY/Y}\mathcal{P}\hat{\otimes}\pi^\flat_\mathcal{YZ/Y}\mathcal{Q}\hat{\otimes} \pi_\mathcal{XZ}^{!(\mathcal{Y})}s^\flat\mathcal{E}^\cdot)\right) &(\text{Projection ; Property }\ref{proj_rel}) \\
&\simeq r_+ \pi_{\mathcal{XZ}+}^{(\mathcal{Y})}(\pi^\flat_\mathcal{XY/Y}\mathcal{P}\hat{\otimes}\pi^\flat_\mathcal{YZ/Y}\mathcal{Q}\hat{\otimes} \pi_\mathcal{XY/Y}^\flat q^{!(\mathcal{Y})}\mathcal{E}^\cdot) & (\text{Property } \ref{compat})\\
&\simeq r_+ \pi_{\mathcal{XZ}+}^{(\mathcal{Y})}(\pi^\flat_\mathcal{YZ/Y}\mathcal{Q}\hat{\otimes}\pi^\flat_\mathcal{XY/Y}(\mathcal{P}\hat{\otimes} q^{!(\mathcal{Y})}\mathcal{E}^\cdot)) &\\
 &\simeq t_+^{(\mathcal{Y})} \pi_{\mathcal{YZ}+/\mathcal{Y}} (\pi^\flat_\mathcal{YZ/Y}\mathcal{Q}\hat{\otimes}\pi^\flat_\mathcal{XY/Y}(\mathcal{P}\hat{\otimes} q^{!(\mathcal{Y})}\mathcal{E}^\cdot))& (\text{Property } \ref{compat})\\
 &\simeq t_+^{(\mathcal{Y})}  (\mathcal{Q}\hat{\otimes}\pi_{\mathcal{YZ}+/\mathcal{Y}}\pi^\flat_\mathcal{XY/Y}(\mathcal{P}\hat{\otimes} q^{!(\mathcal{Y})}\mathcal{E}^\cdot))& (\text{Projection ; Property } \ref{proj})\\
 &\simeq t_+^{(\mathcal{Y})} (\mathcal{Q}\hat{\otimes} u^\flat_{/\mathcal{Y}} p_{+/\mathcal{Y}}(\mathcal{P}\hat{\otimes} q^{!(\mathcal{Y})}\mathcal{E}^\cdot))&(\text{Base change ; Property } \ref{chgmt_base}) \\
&=\Phi_\mathcal{Q}^{\mathcal{Z}\leftarrow\mathcal{Y}}\circ\Phi_\mathcal{P}^{\mathcal{X}\rightarrow\mathcal{Y}}(\mathcal{E}^\cdot).&
\end{align*}

\end{proof}

The most important thing to recall is that the composition of two Fourier-Mukai transforms of kernel $\calP$ and $\calQ$ is a Fourier-Mukai transform of kernel $\calR$ obtained by pulling back $\calP$ and $\calQ$ on $\calX\times\calY\times\calZ$, taking their tensor product, and finally pushing it forward on $\calX\times\calZ$, while making sure to use the good functors regarding the structures of $\Dhz$-modules of $\calP$ and $\calQ$.

\subsection{Fourier-Mukai transform of Poincaré kernel}

From now on, let $\calA$ be a formal abelian scheme over $\calS$ of dimension $d$ and denote by $A_i$ its reduction modulo $\pi^i$. Let also $m$ (resp. $m_i$), $\epsilon$ (resp. $\epsilon_i$) and $\langle-1\rangle$ (resp. $\langle-1\rangle_i$) be respectively the multiplication, neutral element and inverse morphisms of $\calA$ (resp. $A_i$).

%\subsubsection*{$\Dhz$-dual abelian scheme and Poincar\'e sheaf}

%Before defining the Fourier-Mukai transorm of Poincar\'e kernel for $\Dhz$-modules, some objects must be defined.
As the goal of the Fourier-Mukai transform on $\Dhz$-modules is to take into account the structure of $\Dhz$-module of the sheaves over $\calA$, the usual definition of dual abelian scheme as the one representing the Picard functor can't be a good one. The idea is to build another functor, $Pic^\natural$, which will take into account the structure of $\Dhz$-module of the sheaves over $\calA$. Namely, $Pic^\natural(\calA/\calS)$ will be the group of rigid extensions of $\calA$.

This construction and the one of the Poincar\'e sheaf are very similar to what is done in the case of non-arithmetical $\calD$-modules in \cite{Lau}, as it is a summary of the results already proved in \cite{Ma-Me} (chapter I).

\begin{defi}

Let $p_1,p_2:\calA\times\calA\rightarrow\calA$ be the canonical projections. $Pic^\natural(\calA/\calS)$ is defined as the group of sheaves of $\Dhz_{\calA/\calS}$-modules $\calO_\calA$-invertible satisfying the theorem of the square, i.e.
\[m^\flat\calL\simeq p_1^\flat\calL\otimes p_2^\flat\calL,\]
where $m^\flat$ is the inverse image for $\Dhz_{\calA/\calS}$-modules without the shift.

\end{defi}

\begin{prop}

The functor $Pic^\natural(\calA\times\bullet/\bullet):\{$Noetherian $\calS$-schemes $\}\rightarrow\{$groups$\}$ is representable by a smooth group scheme over $\calS$ of dimension $2d$, denoted $\calA^\natural$.

\end{prop}

For the proof, see \cite{Ma-Me}, chapter I. It is important to note that the proof uses the fact that we are looking at sheaves with integrable connection, so that one has to consider $\Dhz$-modules and not $\calD^{(m)}$-modules with $m>0$. This is the main reason why the following results may be harder to prove, if even true, in a more general case.

\begin{defi}

The element of $Pic^\natural(\calA\times\calA^\natural/\calA^\natural)$ associated to the identity $id_{\calA^\natural}\in Hom(\calA^\natural,\calA^\natural)$ is called the Poincar\'e sheaf and denoted by $\calP$.

\end{defi}

\begin{rem}
As it is an element of $Pic^\natural(\calA\times\calA^\natural/\calA^\natural)$, $\calP$ has a structure of $\Dhz_{\calA\times\calA^\natural/\calA^\natural}$-module, hence a structure of relative $\Dhz$-module. That fact implies that $\calA\times\calA^\natural$ and $\calA^\natural$ have to be considered as formal schemes over $\calA^\natural$, so that the Fourier-Mukai transform applied to a $\Dhz_\calA$-module can only give rise to a $\calO_{\calA^\natural}$-module.
\end{rem}

The previous definitions can also be written on the reductions modulo $\pi^i$, leading to $\Dz$-dual group schemes $A_i^\natural$ and Poincar\'e sheaves $\calP_i$. As these definitions are compatible with reductions modulo $\pi^i$, the following result is immediate.

\begin{prop}

There are canonical isomorphisms $\calA^\natural\simeq\varinjlim A_i^\natural$ and $\calP\simeq\varprojlim\calP_i$.

\end{prop}

One can use the representability of $Pic^\natural(\calA\times\bullet/\bullet)$ to find some useful results, such as the following.

\begin{prop}\label{repres}

Given $\calL\in Pic^\natural(\calA\times\calT/\calT)$, there exists only one morphism $f:\calT\rightarrow\calA^\natural$ such that
\[\calL\simeq(id_\calA\times f)^{!(\calT,\calA^\natural)}\calP.\]

In particular, the points of $\calA^\natural$ are in bijection with the $\Dhz_{\calA/\calS}$-modules $\calO_\calA$-invertible verifying the theorem of the square.

Moreover, $(id_\calA\times\epsilon^\natural)^{!(\calA^\natural)}\calP\simeq\calO_\calA$.

\end{prop}

There is also a kind of compatibility result with the dual abelian formal scheme $\calA^\vee$.

\begin{prop}

Define $i : H^0(\calA,\Omega^1_{\calA/\calS})\rightarrow Pic^\natural(\calA/\calS)$ and $p:Pic^\natural(\calA/\calS)\rightarrow Pic^0(\calA/\calS)$ by $i(\omega)=\mathcal{O}_\calA$ with the integrable connection $d+\omega$ and $p(\mathcal{L})=\mathcal{L}$ seen as a $\mathcal{O}_\calA$-module.
Then there is an exact sequence of abelian schemes:
\[E(\calA/\calS) = (0\rightarrow H^0(\calA,\Omega^1_{\calA/\calS})\rightarrow Pic^\natural(\calA/\calS)\rightarrow Pic^0(\calA/\calS)).\]
Moreover, if $\calS$ is affine, the morphism $p$ is surjective.

One then has a functor $\calT\mapsto E(\calA\times \calT/\calT):\{$Locally noetherian $\calS$-schemes$\}\rightarrow\{$Exact sequences of abelian schemes$\}$. This functor is representable by an exact sequence of abelian schemes over $\calS$:

%Il existe un morphisme de variétés abéliennes sur $S$, noté $\theta:A^\natural\rightarrow A^\vee$, et une suite exacte de $S$-schémas en groupes abéliens
\[0\rightarrow\mathbb{V}(\epsilon^*\mathcal{T}_{\calA/\calS})\rightarrow \calA^\natural\rightarrow \calA^\vee\rightarrow0,\]
whith $\mathbb{V}(\mathcal{E})=Spec(Sym(\mathcal{E}))$.
Denote by $\theta$ the morphism from $\calA^\natural$ to $\calA^\vee$.

\end{prop}

\begin{proof}

This is proved on the reductions modulo $\pi^i$, using results of sections 2.6 and 3.2.3 of chapter I of \cite{Ma-Me}. Then simply take the inverse limit.

\end{proof}

\begin{prop}\label{corresp_P}
There is an isomorphism of $\mathcal{O}_{\calA\times \calA^\natural}$-modules
\[\mathcal{P}\simeq(id_\calA\times \theta)^*\mathcal{P}^\vee,\]
where $\mathcal{P}^\vee$ is the Poincar\'e sheaf on $\calA\times \calA^\vee$.

\end{prop}

\begin{proof}

The representability of $\calT\mapsto(Pic^\natural(\calA\times\calT/\calT)\rightarrow Pic^0(\calA\times\calT/\calT)$ leads to the following commutative diagram
\[\xymatrix@=40pt{
		Pic^\natural(\calA\times \calA^\natural/\calA^\natural) \ar^{\theta_{\calA^\natural}}[r] & Pic^0(\calA\times \calA^\natural/\calA^\natural) \\
		Hom(\calA^\natural,\calA^\natural) \ar^{\Phi^\natural_{\calA^\natural}}[u] \ar_{\theta\circ\bullet}[r] & Hom(\calA^\natural,\calA^\vee) \ar_{\Phi^0_{\calA^\natural}}[u]
}\]
where $\theta_{\calA^\natural}$ is the forgetting of the structure of $\mathcal{D}_{\calA\times \calA^\natural/\calA^\natural}$-module and $\Phi^\natural_{\calA^\natural}$ and $\Phi^0_{\calA^\natural}$ are the representation isomorphisms.
Evaluating in the element $id_{\calA^\natural}\in Hom(\calA^\natural,\calA^\natural)$, one obtains
\[\xymatrix@=40pt{
		\mathcal{P} \ar@{|->}[r] & \mathcal{P} \\
		id_{\calA^\natural} \ar@{<->}[u] \ar@{|->}[r] & \theta \ar@{<->}[u]
}\]
where the left $\mathcal{P}$ is a $\mathcal{D}$-module and the right one is a $\mathcal{O}$-module.

On the other hand, as $\mathcal{P}^\vee=\Phi^0_{\calA^\vee}(id_{\calA^\vee})$, Yoneda's lemma states that 
\[\Phi^0_{\calA^\natural}(\theta)=\left(Pic^0(\calA\times\bullet/\bullet)(\theta)\right)(\mathcal{P}^\vee)=(id_\calA\times \theta)^*\mathcal{P}^\vee.\]
Hence, $\mathcal{P}=\Phi^0_{\calA^\natural}(\theta)=(id_\calA\times \theta)^*\mathcal{P}^\vee$, up to an isomorphism and as $\mathcal{O}$-modules.

\end{proof}

\begin{prop}

The morphism $\theta:\calA^\natural\rightarrow\calA^\vee$ is a torsor, hence affine.

Moreover, the sheaf $\theta_*\calO_{\calA^\natural}$ has an exhaustive filtration by locally free sub-$\calO_{\calA^\vee}$-modules of finite rank which makes it a filtrated $\calO_{\calA^\vee}$-algebra, with an isomorphism of graduated $\calO_{\calA^\vee}$-algebras
\[\varphi_\bullet:gr_\bullet(\theta_*\calO_{\calA^\natural}) \rightarrow \pi^{\vee*}Sym_\bullet^{\mathcal{O}_\calS}(\epsilon^*\mathcal{T}_\mathcal{A/S}).\]

\end{prop}

\begin{proof}

It is proved in the same way as in \cite{Lau}, part 2.3, for the reductions modulo $\pi^i$ and one can then take the inverse limit of these results modulo $\pi^i$.

\end{proof}

%\subsubsection*{Fourier-Mukai transform}

The Fourier-Mukai transform on $\calA$ for $\Dhz$-modules is defined as the Fourier-Mukai transform with kernel the Poincar\'e (relative) sheaf.

\begin{defi}

Denote by $p:\calA\times\calA^\natural\rightarrow\calA$ and $p^\natural:\calA\times\calA^\natural\rightarrow\calA^\natural$ the canonical projections.

The Fourier-Mukai transform on $\calA$ for $\Dhz$-modules is defined as the functor $\calF:D^b_{qcoh}(\Dhz_{\calA/\calS})\rightarrow D^b_{qcoh}(\calO_{\calA^\natural})$ given by
\[\calF(\calE^\cdot)=\Phi_\calP^{\calA\rightarrow\calA^\natural}(\calE^\cdot)=p^\natural_{+/\calA^\natural}(\calP\overset{\mathbb{L}}{\otimes}_{\calA\times\calA^\natural}p^{!(\calA^\natural)}\calE^\cdot).\]

The $\calD$-dual Fourier-Mukai transform is then defined as the functor $\calF^\natural:D^b_{qcoh}(\calO_{\calA^\natural})\rightarrow D^b_{qcoh}(\Dhz_{\calA/\calS})$ given by
\[\calF^\natural(\calE^\cdot)=\Phi_\calP^{\calA\leftarrow\calA^\natural}(\calE^\cdot)=p^{(\calA^\natural)}_+(\calP\overset{\mathbb{L}}{\otimes}_{\calA\times\calA^\natural}p^{\natural\flat}_{\calA^\natural}\calE^\cdot).\]

\end{defi}

\subsection{Essential surjectivity of the Fourier-Mukai transform for $\Dhz$-modules}

As stated in the introduction, the goal of this Fourier-Mukai transform is to be involutive and hence to give an equivalence of categories between $D^b_{qcoh}(\Dhz_{\calA/\calS})$ and $D^b_{qcoh}(\calO_{\calA^\natural})$. However, the proof of the involutivity in the non-arithmetic case, which can be found in \cite{Lau} for instance, will not fully apply in the arithmetic case, mostly because of the lack of Kashiwara's theorem for $\Dhz$-modules.

Still, it is good to have an idea of the proof to know where we are going, so let us recall here the main steps. First, it is known that the composition of two Fourier-Mukai transform is a Fourier-Mukai transform, hence $\calF^\natural\circ\calF$ (or $\calF\circ\calF^\natural$) is a Fourier-Mukai transform of kernel $\calR$ from the category $D^b_{qcoh}(\Dhz_{A/S})$ (or $D^b_{qcoh}(\calO_{A^\natural})$) into itself. The goal is to prove that the Fourier-Mukai transform of kernel $\calR$ is an isomorphism, and hence understand the kernel itself. After some transformations using the base change theorem, the seesaw principle and some basic properties of the elements of $Pic^\natural(A\times A^\natural/A^\natural)$, $\calR$ is proved to be isomorphic to the pullback by the multiplication of $p_+^{(A^\natural)}\calP$ (or $p^\natural_{+/A^\natural}\calP$). A very difficult lemma states that this sheaf is isomorphic to the push-forward of $\calO_S$ by the neutral element, modulo a shift. Once this lemma proven, the end of the proof is straightforward using base change and projection formulas.

%So, the key argument is the understanding of the push-forward of $\calP$ by the two projections, hence the following lemma.

So, the key argument is the understanding of the push-forward of $\calP$ by the two projections. The main idea will be to adapt the proofs in \cite{Lau} to the arithmetic case. If one of the two push-forward is computable using some more arguments than in the non-arithmetic case, the other one is more complicated to compute, hence the following lemma.

\begin{lem}

%\begin{enumerate}

	 $p_{+/\calA^\natural}^\natural\calP\simeq\epsilon^\natural_*\calO_\calS[-d]\in D^b_{qcoh}(\calO_{\calA^\natural})$.
	
%	\item $\Ddag_{\calA/\calS,\Q}\otimes_{\Dhz_{\calA/\calS,\Q}}p_+^{(\calA^\natural)}\calP_\Q\simeq\epsilon_+\calO_{\calS,\Q}[-d]\in D^b_{qcoh}(\Ddag_{\calA/\calS,\Q})$.

%\end{enumerate}

\end{lem}

\begin{rem}

The expected second isomorphism should be $p_+^{(\calA^\natural)}\calP\simeq\epsilon_+\calO_{\calS}[-d]\in D^b_{qcoh}(\Dhz_{\calA/\calS})$, however it can't be proved using the same method as in \cite{Lau} (or, at least, not without a lot of work). The main reason is that the proof of this isomorphism in the non-arithmetic case uses the Kashiwara's equivalence theorem, which is known to be false for $\Dhz$-modules. However, as there exists a version of this result for $\Ddag_\Q$-modules (see e.g. \cite{Caro09}), one could expect an isomorphism of the form
\[\Ddag_{\calA/\calS,\Q}\otimes_{\Dhz_{\calA/\calS,\Q}}p_+^{(\calA^\natural)}\calP_\Q\simeq\epsilon_+\calO_{\calS,\Q}[-d]\in D^b_{qcoh}(\Ddag_{\calA/\calS,\Q}).\]
But even if this isomorphism were true it will not be of use for the proof of the involutivity of the Fourier-Mukai transform, hence we will not try to prove it.

\end{rem}

\begin{proof}

%\begin{enumerate}

In order to prove the isomorphism, the first thing to do is to find a morphism of $\calO_{\calA^\natural}$-modules between $p^\natural_{+/\mathcal{A}^\natural}\mathcal{P}$ and $\epsilon_*^\natural\mathcal{O}_{\mathcal{S}}[-d]$.
Note that, as $\calP$ is invertible, hence coherent, and $p^\natural$ and $\epsilon^\natural$ are proper, the sheaves $p^\natural_{+/\mathcal{A}^\natural}\mathcal{P}$ and $\epsilon^\natural_* \mathcal{O}_\mathcal{S}$ are coherent.

The goal is to construct the morphism $p^\natural_{+/\mathcal{A}^\natural}\mathcal{P}\rightarrow\epsilon_*^\natural\mathcal{O}_{\mathcal{S}}[-d]$ by adjunction from a morphism $\epsilon^{\natural*}p^\natural_{+/\mathcal{A}^\natural}\mathcal{P}\rightarrow\mathcal{O}_{\mathcal{S}}[-d]$ (recall that the complex of $\mathcal{O}_\mathcal{S}$-modules $\epsilon^{\natural*}\mathcal{E}^\cdot$ is defined as $R\varprojlim \epsilon^{\natural*}_i\mathcal{E}_i^\cdot$).
Consider the following base chance diagram
\[\xymatrix@=40pt{
\mathcal{A} \ar@{^(->}^{id_\mathcal{A}\times\epsilon^\natural}[r] \ar_\pi[d] & \mathcal{A}\times \mathcal{A}^\natural \ar^{p^\natural}[d]\\
\mathcal{S} \ar@{^(->}^{\epsilon^\natural}[r] & \mathcal{A}^\natural.
}\]

Using the base change formula for $\calO$-modules in the previous diagram, one finds the following isomorphisms.

\[\epsilon^{\natural*} p^\natural_{+/\mathcal{A}^\natural}\mathcal{P}
= \epsilon^{\natural*} Rp^\natural_{*}(\Dhz_{\mathcal{A}^\natural\leftarrow\mathcal{A}\times\mathcal{A}^\natural/\mathcal{A}^\natural}\otimes_{\Dhz_{\mathcal{A}\times\mathcal{A}^\natural/\mathcal{A}^\natural}}\mathcal{P})
\simeq R\pi_* (id_\mathcal{A}\times\epsilon^\natural)^{*}(\Dhz_{\mathcal{A}^\natural\leftarrow\mathcal{A}\times\mathcal{A}^\natural/\mathcal{A}^\natural}\otimes_{\Dhz_{\mathcal{A}\times\mathcal{A}^\natural/\mathcal{A}^\natural}}\mathcal{P}).\]
Denote by $DR_{\mathcal{A}\times \mathcal{A}^\natural/\mathcal{A}^\natural}(\mathcal{P})$ the de Rham complex of $\mathcal{P}$ (between degrees $0$ and $d$):
\[DR_{\mathcal{A}\times \mathcal{A}^\natural/\mathcal{A}^\natural}(\mathcal{P}) : 0\rightarrow\mathcal{P}\rightarrow\Omega^1_{\mathcal{A}\times\mathcal{A}^\natural/\mathcal{A}^\natural}\otimes\mathcal{P}\rightarrow\ldots\rightarrow\Omega^{d}_{\mathcal{A}\times\mathcal{A}^\natural/\mathcal{A}^\natural}\otimes\mathcal{P}\rightarrow0\]

The transfer bimodule $\Dhz_{\mathcal{A}^\natural\leftarrow \mathcal{A}\times \mathcal{A}^\natural/\calA^\natural}$ is canonically isomorphic to its Spencer complex
\[0\rightarrow\Dhz_{\calA\times\calA^\natural/\calA^\natural}\rightarrow\Omega_{\calA\times\calA^\natural/\calA^\natural}\otimes \Dhz_{\calA\times\calA^\natural/\calA^\natural}\rightarrow\ldots\rightarrow\Omega_{\calA\times\calA^\natural/\calA^\natural}^{dim(\calA)}\otimes \Dhz_{\calA\times\calA^\natural/\calA^\natural}\rightarrow0,\]
hence there is a canonical isomorphism
\[\Dhz_{\mathcal{A}^\natural\leftarrow\mathcal{A}\times\mathcal{A}^\natural/\mathcal{A}^\natural}\otimes_{\Dhz_{\mathcal{A}\times\mathcal{A}^\natural/\mathcal{A}^\natural}}\mathcal{P}\simeq DR_{\mathcal{A}\times \mathcal{A}^\natural/\mathcal{A}^\natural}(\mathcal{P})[d],\]
where the second complex is then between degrees $-d$ and $0$.
There is then an isomorphism
\[\epsilon^{\natural*} p^\natural_{+/\mathcal{A}^\natural}\mathcal{P}\simeq R\pi_* (id_\mathcal{A}\times\epsilon^\natural)^{*}DR_{\mathcal{A}\times \mathcal{A}^\natural/\mathcal{A}^\natural}(\mathcal{P})[d].\]
As the elements of the de Rham complex of $\calP$ are flats, they are acyclic for $(id_\mathcal{A}\times\epsilon^\natural)^{*}$ and $(id_\mathcal{A}\times\epsilon^\natural)^{*}\Omega^k_{\mathcal{A}\times\mathcal{A}^\natural/\mathcal{A}^\natural}\simeq\Omega^k_{\mathcal{A/S}}$. Hence,
\[(id_\mathcal{A}\times\epsilon^\natural)^{*}DR_{\mathcal{A}\times \mathcal{A}^\natural/\mathcal{A}^\natural}(\mathcal{P})\simeq DR_{\mathcal{A/S}}((id_\mathcal{A}\times\epsilon^\natural)^{*}\mathcal{P}).\]
%\[\epsilon^{\natural\flat} p^\natural_{+/\mathcal{A}^\natural}\mathcal{P}\simeq R\varprojlim L\epsilon_i^{\natural*} Rp_{i*}^\natural DR_{A_i\times A_i^\natural/A_i^\natural}(\mathcal{P}_i).\]
%Le théorème du changement de base pour les $\mathcal{O}$-modules (\ref{chgmt_base_O} p.\pageref{chgmt_base_O}) donne alors 
%\[\epsilon^{\natural\flat} p^\natural_{+/\mathcal{A}^\natural}\mathcal{P}\simeq R\varprojlim R\pi_{i*} L(id_{A_i}\times\epsilon_i^{\natural})^* DR_{A_i\times A_i^\natural/A_i^\natural}(\mathcal{P}_i).\]
%Comme les termes du complexe $DR_{A_i\times A_i^\natural/A_i^\natural}(\mathcal{P}_i)$ sont plats, ils sont acycliques pour $ L(id_{A_i}\times\epsilon_i^{\natural})^*$ et
%\[L(id_{A_i}\times\epsilon_i^{\natural})^* DR_{A_i\times A_i^\natural/A_i^\natural}(\mathcal{P}_i)\simeq DR_{A_i/S_i}\left((id_{A_i}\times\epsilon_i^\natural)^*\mathcal{P}_i\right).\]

As previously settled (property \ref{repres}), the $\calO_\calA$-module $(id_\mathcal{A}\times\epsilon^\natural)^{*}\mathcal{P}$ is well known: in fact, as $\calD_{\calA/\calS}$-modules, $(id_\mathcal{A}\times\epsilon^\natural)^{!(\calA^\natural)}\mathcal{P}\simeq\calO_\calA$. Their underlying $\calO_\calA$-modules are then isomorphic too.

There is then a canonical isomorphism
\[\epsilon^{\natural*} p^\natural_{+/\mathcal{A}^\natural}\mathcal{P}\simeq R\pi_* DR_\mathcal{A/S}(\mathcal{O}_\mathcal{A})[d]\simeq R\pi_{*} DR(\mathcal{A/S})[d],\]
which is then a complex in degrees between $-d$ and $0$.

As the $2d$-th cohomology group of $R\pi_* DR(\mathcal{A/S})$ is $\mathcal{H}^{2d}_{dR}(\mathcal{A/S})$, it will be usefull to understand the cohomology of $\calA$.

\begin{lem}\label{cohom}
\[H^n_{dR}(\mathcal{A}/\mathcal{S})\simeq \varprojlim H^n_{dR}(A_i/S_i)\simeq \bigwedge^n Lie(\mathcal{A}^\natural).\]
\end{lem}

\begin{proof}

It is known that the functors $R\Gamma$ and $R\varprojlim$ commute, hence
\[R\Gamma R\varprojlim DR(A_i/S_i) \simeq R\varprojlim R\Gamma DR(A_i/S_i).\]
%où $DR(A_i/S_i)$ désigne le complexe de de Rham de $A_i$.

As the elements of $DR(A_i/S_i)$ are the reductions modulo $\pi^i$ of the elements of $DR(A_{i+1}/S_{i+1})$, $R\varprojlim DR(A_i/S_i)$ is the complex which elements of the projective limits of the elements of the $DR(A_i/S_i)$, i.e. $DR(\calA/\calS)$. Hence, the cohomology groups of the left hand side are the $H^n_{dR}(\calA/\calS)$. It is then sufficient to compute the cohomology groups of the right hand side.

A priori, the spectral sequence of the composed functors gives two terms for the $n$-th cohomology group of the right hand side: $\varprojlim H^n_{dR}(A_i/S_i)$ and $R^1\varprojlim H^{n-1}_{dR}(A_i/S_i)$, the functor $R\varprojlim$ being of cohomological dimension 1.
However, it is known that $H^n_{dR}(A_i/S_i)\simeq\epsilon^{\natural*}\Omega^n_{A_i^\natural/S_i}$ (\cite{Coleman}, theorem 2.2), hence the $R^1\varprojlim$ is trivial and that conclude the proof of the first isomorphism.

The second isomorphism is a direct consequence of the isomorphism $H^1(A_i/S_i)\simeq Lie(A_i^\natural)$ (see \cite{Ma-Me}, parts I.4.1 and I.4.2).

\end{proof}

As $H_{dR}^1(\mathcal{A/S})\simeq Lie(\mathcal{A}^\natural)$ is a free $\mathcal{V}$-module of rank $2d$, it is clear that $H^{2d}_{dR}(\mathcal{A/S})\simeq \mathcal{V}$.

%Après tensorisation par $K$, $\mathcal{H}^{2d}_{dR}(\mathcal{A/S})_\mathbb{Q}=\mathcal{H}^{2d}_{rig}(A/S)$. Comme la cohomologie rigide est une cohomologie de Weil, $\mathcal{H}^{2d}_{rig}(A/S)=\bigwedge^{2d}\mathcal{H}^1_{rig}(A/S)$ qui est isomorphe par le morphisme trace à $K=\mathcal{O}_{\mathcal{S},\mathbb{Q}}$.

Hence, denoting $F^\bullet=R\pi_*DR(\mathcal{A}/\mathcal{S})[d]$, one finds a morphism of complexes
\[\xymatrix@=30pt{
0 \ar[r] & F^{-d} \ar[r]\ar[d] & \ldots \ar[r] & F^{d-1} \ar^{d_{d-1}}[r] \ar[d] & F^{d} \ar[r] \ar[d] & 0 \\
0 \ar[r] & 0 \ar[r] & \ldots \ar[r] & 0 \ar[r] & \faktor{F^{d}}{Im(d_{d-1})} \ar@{=}[d] \ar[r] & 0 \\
& & & & \mathcal{H}^{d}(F^\bullet)\simeq\mathcal{O}_{\mathcal{S}},&\\
}\]
i.e. a morphism
\[\epsilon^{\natural*} p_{+/\mathcal{A}^\natural}^\natural\mathcal{P}\rightarrow\mathcal{O}_{\mathcal{S}}[-d].\]
By adjunction, one finds the expected morphism
\[p_{+/\mathcal{A}^\natural}^\natural\mathcal{P}\rightarrow\epsilon_*^\natural\mathcal{O}_{\mathcal{S}}[-d].\]

%Pour montrer qu'il s'agit d'un isomorphisme, on procède par étapes. On étudie d'abord les points de $\mathcal{A}^\natural$ (i.e. les éléments de $Pic^\natural(\mathcal{A/S})$), puis on montre le résultat au-dessus de chaque point de $\mathcal{A}^\natural$ avant de conclure grâce à la cohérence des faisceaux.

The goal is then to prove it is an isomorphism. In order to do that we will proceed as follow: First we will understand more precisely the points of $\calA^\natural$, then prove the expected isomorphism on every point of $\calA^\natural$, next we will use the coherence of the sheaves and Nakayama's lemma to prove the isomorphism on the varieties $A_i^\natural$, and finally deduce the expected isomorphism on the formal scheme $\calA$.

\underline{Step 1:} Recall that the points of $\calA^\natural$ are the elements of $Pic^\natural(\mathcal{A/S})$. One have the following lemma.
\begin{lem}

For all $\mathcal{L}\in Pic^\natural(\mathcal{A/S})$
\[H^{n}_{dR}(\mathcal{L})\simeq\left\{\begin{matrix} 0 & \text{if }\mathcal{L}\not\simeq\mathcal{O}_\mathcal{A}\\
L_{2d-n}\epsilon^{\natural*}\epsilon_*^\natural \mathcal{O}_{\mathcal{S}} & \text{if }\mathcal{L}\simeq\mathcal{O}_\mathcal{A}\end{matrix}\right.\]

\end{lem}
\begin{proof}

%On a déjà montré un lemme similaire dans le cas non-arithmétique, donc on sait déjà que pour tout $\mathcal{L}_i\in Pic^\natural(A_i/S_i)$ différent de $\mathcal{O}_{A_i}$, $H^n_{dR}(\mathcal{L}_i)=0$ pour tout entier $n$ et
%\[H^n_{dR}(A_i/S_i)\simeq L_{2d-n}\epsilon_i^{\natural*}\epsilon_{i*}^\natural\mathcal{O}_{S_i}.\]

%On sait aussi que les faisceaux inversibles sur une variété abélienne $A$ sont caractérisés par $H^1(\mathcal{O}_A^*)$, or les inversibles de $\mathcal{O}_\mathcal{A}$ et des $\mathcal{O}_{A_i}$ sont les mêmes que ceux de $\mathcal{O}_{A_1}$. En particulier, si $\mathcal{L}$ est un faisceau $\mathcal{O}_\mathcal{A}$-inversible tel que pour un entier $i$ $\mathcal{L}_i\simeq\mathcal{O}_{A_i}$ alors $\mathcal{L}\simeq\mathcal{O}_\mathcal{A}$.

%Ce faisant, on va faire commuter le lemme à la limite projective. Pour ce faire, on remarque dans un premier temps que comme $\mathcal{L}$ est inversible, les termes du complexe de de Rham de $\mathcal{L}$ sont acycliques pour la limite projective. Ainsi,
%\[DR(\mathcal{L})\simeq\varprojlim DR(\mathcal{L}_i).\]

Take $\mathcal{L}\in Pic^\natural(\mathcal{A/S})$ such that $\mathcal{L}\not\simeq\mathcal{O}_\mathcal{A}$. Denote $H^0=H^0_{dR}(\mathcal{L}_\Q)$, which is then a $K=Frac(V)$-vector space and denote $r$ its dimension. It is known (\cite{Kedlaya}, lemma 5.1.5) that $r=0$ or $r=1$ and there is an exact sequence
\[0\rightarrow \mathcal{O}_{\mathcal{A},\Q}\otimes_K H^0 \rightarrow \mathcal{L}_\Q \rightarrow \mathcal{L}' \rightarrow 0,\]
where $\mathcal{L}'$ is a coherent $\mathcal{O}_{\mathcal{A},\Q}$-module with an integrable connection.

If we denote by $sp:\calA_K\rightarrow\calA$ the specialization morphism, it is known that the functors $sp^*$ and $sp_*$ achieve an equivalence of categories between the coherent $\mathcal{O}_{\mathcal{A},\Q}$-modules and the coherent $\mathcal{O}_{\mathcal{A}_K}$-modules. As settled in \cite{Coh_rig_1} (proposition 2.2.3) $\mathcal{L}'$ is then locally free of finite rank. Denote $r'$ its rank.
Hence, as $\mathcal{L}_\Q$ is invertible, $r+r'=1$. However, $r'\neq0$ because if not there will be an isomorphism $\mathcal{L}_\Q\simeq\mathcal{O}_{\mathcal{A},\Q} \otimes H^0 \simeq\mathcal{O}_{\mathcal{A},\Q}$. Hence $r'=1$ and $r=0$, then $H^0=0$.

\noindent We then know that $H^0_{dR}(\mathcal{L}_\Q)=0$, hence $H^0_{dR}(\mathcal{L})$ is of torsion. However, as $H^0_{dR}(\mathcal{L})\subset H^0(\mathcal{L})$ and none of the sections of $\calL$ can be of torsion over $\calA$ (because locally $\mathcal{L}\simeq\mathcal{O}_\mathcal{A}$), $H^0_{dR}(\mathcal{L})$ is torsion-free. Hence $H^0_{dR}(\mathcal{L})=0$.

Assume then there is an integer $n>0$ such that $H^n_{dR}(\mathcal{L})\neq0$ and consider the smallest one. As $\mathcal{L}\in Pic^\natural(\mathcal{A/S})$, it verifies the theorem of the square: $m^*\mathcal{L}\simeq p_1^*\mathcal{L} \otimes p_2^*\mathcal{L}$. Moreover, there exists an isomorphism $\varphi=(m,\langle-1\rangle\circ p_2)$ of $\calA\times\calA$ ($\varphi^2=id$) such that $m=p_1\circ\varphi$, hence $\varphi^*m^*\mathcal{L}\simeq p_1^*\mathcal{L}$. The K\"unneth formula then settles
\[H^n_{dR}(m^*\mathcal{L})\simeq\bigoplus_{k+l=n} H^k_{dR}(\mathcal{L})\otimes H^l_{dR}(\mathcal{L})\simeq\bigoplus_{k+l=n} H^k_{dR}(\mathcal{L})\otimes H^l_{dR}(\mathcal{O}_{\mathcal{A}}).\]
As for all $k<n$, $H^k_{dR}(\mathcal{L})=0$, one deduces that
\[0\simeq H^n_{dR}(\mathcal{L})\otimes H^0_{dR}(\mathcal{O}_{\mathcal{A}})\simeq H^n_{dR}(\mathcal{L}),\]
which is a contradiction. Hence for all $n$, $H^n_{dR}(\mathcal{L})=0$.

Let now assume $\mathcal{L}\simeq\mathcal{O}_\mathcal{A}$.
The goal is to compute

\[H^n_{dR}(\mathcal{A/S})\simeq R\varprojlim H^n_{dR}(A_i/S_i) \simeq R\varprojlim \bigwedge^n Lie(A_i^\natural).\]

In order to do that, denote by $\mathcal{I}$ the ideal associated to the injection $\epsilon^\natural_i$, $\mathcal{I}=V(x_1,...,x_{2d})$ and $R=\bigoplus_{i=1}^{2d}\calO_{A_i^\natural}e_i$ and consider the Koszul complex
\[0\rightarrow\bigwedge^{2d} R\rightarrow\ldots\rightarrow\bigwedge^2 R\rightarrow R\rightarrow\mathcal{O}_{A^\natural_i}\rightarrow\epsilon^\natural_* \calO_{S_i}\rightarrow0,\]
where the arrow $R\rightarrow\mathcal{O}_{A^\natural_i}$ is given by $e_n\mapsto x_n$. Applying $Hom(\bullet,\epsilon_{i*}^\natural\calO_{S_i})$, one find the complex whose elements are the $\epsilon_{i*}^\natural\calO_{S_i}\otimes\left(\bigwedge^n\faktor{\mathcal{I}}{\mathcal{I}^2}\right)^\vee$, hence the $\epsilon_{i*}^\natural\calO_{S_i}\otimes\bigwedge^n Lie(A_i^\natural)$, and whose differentials are zero. Hence, the $n$-th cohomology group of this complex is
\[\calE xt^n(\epsilon_{i*}^\natural\calO_{S_i},\epsilon_{i*}^\natural\calO_{S_i})\simeq\epsilon_{i*}^\natural\calO_{S_i}\otimes\bigwedge^n Lie(A_i^\natural),\]

then $H^n_{dR}(A_i/S_i)\simeq \mathcal{H}^n((\epsilon_i^\natural)^{-1}R\mathcal{H}om_{\calO_{A_i}}(\epsilon_{i*}^\natural\calO_{S_i},\epsilon_{i*}^\natural\calO_{S_i}))$. Hence, by adjunction
\[H^n_{dR}(A_i/S_i)\simeq\mathcal{H}^n(R\mathcal{H}om_{\calO_{S_i}}(L\epsilon^{\natural*}_{i}\epsilon_{i*}^\natural\calO_{S_i},\calO_{S_i}))\simeq L_{2d-n}\epsilon^{\natural*}_i\epsilon^\natural_{i*}\calO_{S_i}.\]

As these isomorphisms are compatible to the reductions modulo $\pi^i$, one has the expected isomorphism
\[H^n_{dR}(\mathcal{A/S})\simeq R\varprojlim H^n_{dR}(A_i/S_i) \simeq L_{2d-n}\epsilon^{\natural*}\epsilon^\natural_*\mathcal{O}_\mathcal{S}.\]

\end{proof}

\underline{Step 2:} As the goal is to use Nakayama's lemma on the schemes $A_i^\natural$, we will need to have an isomorphism with restriction of each closed point of each $A_i^\natural$. We will see that such a point can be viewed as a $Spf(W(k(x)))$-point of $\calA^\natural$. This is the reason of the hypothesis in the following lemma.

\begin{lem}
Let $k$ be a finite extension of the residue field of $V$, $W(k)$ its ring of Witt vectors and $\mathcal{S}'=Spf(W(k))$.
For all $\mathcal{S}'$-point of $\mathcal{A}^\natural$, if $\iota:\mathcal{S}'\hookrightarrow \mathcal{A}^\natural$ is the injection, then

\[L\iota^* p^\natural_{+/\mathcal{A}^\natural}\mathcal{P}\simeq L\iota^*\epsilon_*^\natural\mathcal{O}_\mathcal{S}[-d].\]
\end{lem}

\begin{proof}

First of all, remark that the morphism
\[L\iota^* p^\natural_{+/\mathcal{A}^\natural}\mathcal{P}\rightarrow L\iota^*\epsilon_*^\natural\mathcal{O}_\mathcal{S}[-d].\]
is well defined, as it is the pullback by $\iota$ of the morphism $p^\natural_{+/\mathcal{A}^\natural}\mathcal{P}\rightarrow\epsilon_*^\natural\mathcal{O}_\mathcal{S}[-d]$ previously built.

Also note that, using some base change in the diagram
\[\xymatrix{
\mathcal{S}' \ar[r] \ar@{^(->}[d]_{{\epsilon^\natural}'} & \mathcal{S} \ar@{^(->}[d]^{\epsilon^\natural} \\
{\mathcal{A}^\natural}' \ar[r]_{{p^\natural}'} & \mathcal{A}^\natural
}\]
one can assume $\calS'=\calS$.

In that case, as $\epsilon^\natural_* \mathcal{O}_\mathcal{S}$ has support in $\epsilon^\natural(\mathcal{S})$, it is sufficient to prove that
\[\left\{
\begin{matrix}
\forall \iota:\mathcal{S}\hookrightarrow\mathcal{A}^\natural,\iota\neq\epsilon^\natural & L\iota^*p^\natural_{+/\mathcal{A}^\natural}\mathcal{P}=0. \\
& L\epsilon^{\natural*}p^\natural_{+/\mathcal{A}^\natural}\mathcal{P}\simeq L\epsilon^{\natural*}\epsilon^\natural_* \mathcal{O}_\mathcal{S} [-d].
\end{matrix}
\right.\]

The base change formula in the diagram
\[\xymatrix@=40pt{
\mathcal{A} \ar@{^(->}^{id_\mathcal{A}\times\iota}[r] \ar_\pi[d] & \mathcal{A}\times \mathcal{A}^\natural \ar^{p^\natural}[d]\\
\mathcal{S} \ar@{^(->}_{\iota}[r] & \mathcal{A}^\natural
}\]
gives the isomorphism
\[L\iota^* p^\natural_{+/\mathcal{A}^\natural}\mathcal{P}\simeq\pi_+ L(id_\mathcal{A}\times\iota)^*\mathcal{P}.\]
However, $L(id_\mathcal{A}\times\iota)^*\mathcal{P}=(id_\mathcal{A}\times\iota)^*\mathcal{P}$ is by definition of $\mathcal{P}$ the element of $Pic^\natural(\mathcal{A/S})$ which corresponds to $\iota$. Denote by $\mathcal{L}$ this element. As $\pi_+\mathcal{L}\simeq H^\bullet_{dR}(\mathcal{L})[d]$, the previous lemma states that
\[L\iota^* p^\natural_{+/\mathcal{A}^\natural}\mathcal{P}\simeq\left\{\begin{matrix}
0 &\text{if }\mathcal{L}\not\simeq\mathcal{O}_\mathcal{A} & \text{i.e. } \iota\neq \epsilon^\natural. \\
L\epsilon^{\natural*}\epsilon^\natural_* \mathcal{O}_\mathcal{S}[-d] & \text{if }\mathcal{L}\simeq\mathcal{O}_\mathcal{A} & \text{i.e. } \iota=\epsilon^\natural.
\end{matrix}\right.\]
Hence the isomorphism
\[L\iota^*p^\natural_{+/\mathcal{A}^\natural}\mathcal{P}\simeq L\iota^*\epsilon^\natural_* \mathcal{O}_\mathcal{S}[-d].\]

%A présent, comme $p^\natural_{+/A^\natural}\mathcal{P}\simeq Rp^\natural_* DR_{A\times A^\natural/A^\natural}(\mathcal{P})$ et que 
%\[Rp^\natural_*(\Omega^i_{A\times A^\natural/A^\natural}\otimes\mathcal{P})\simeq L\theta^*\mathcal{F}^\vee(\Omega^i_{A/S})\]
%est cohérent (avec $\mathcal{F}^\vee$ la transformée de Fourier-Mukai sur les $\mathcal{O}$-modules, qui envoit les cohérents sur les cohérents), $p^\natural_{+/A^\natural}\mathcal{P}$ est un complexe de faisceaux cohérents. De plus $A^\natural$ est localement noethérien, donc le lemme de Nakayama assure que
%\[p^\natural_{+/A^\natural}\mathcal{P}\simeq\epsilon^\natural_*k[-2d],\]
%sur $S=Spec(k)$. Ce qui conclut le lemme intermédiaire.

\end{proof}

\underline{Step 3:} To conclude, we will use the coherence of the sheaves to obtain an isomorphism on the schemes $A_i^\natural$. In order to do so, remark that the morphism
\[p^\natural_{+/\mathcal{A}^\natural}\mathcal{P}\rightarrow\epsilon_*^\natural\mathcal{O}_\mathcal{S}[-d]\]
induces a morphism
\[p^\natural_{+/A_i^\natural}\mathcal{P}_i\rightarrow\epsilon_{i*}^\natural\mathcal{O}_{S_i}[-d]\]
on the reductions modulo $\pi^i$.
Fix a closed point $x$ of $A_i^\natural$, which is then a closed point of $A_0^\natural$. Denote $k$ its residue field, which is then a finite extension of the residue field of $V$. As $k$ is finite, it is perfect and one can take its Witt vectors ring $W(k)$ and denote by $\omega$ its uniformizer. One then has the following diagram
\[\xymatrix@=40pt{
Spec\left(\faktor{W(k)}{\omega}\right) \ar@{^(->}[r] \ar@{^(->}[d] \ar@{^(->}[rd] & A_0^\natural \ar@{^(->}[d]\\
Spec\left(\faktor{W(k)}{\omega^2}\right) \ar@{^(-->}[r]_f & A_1^\natural
}\]
with $f$ that exists by formal smoothness. Doing an immediate recurrence, one find the closed point $x$ can be lift up on all the varieties $A_i^\natural$. Denote by $\iota:Spf(W(k))\hookrightarrow\mathcal{A}^\natural$ this lift. The previous lemma then settles that the sheaves $p^\natural_{+/\mathcal{A}^\natural}\mathcal{P}$ and $\epsilon_*^\natural\mathcal{O}_\mathcal{S}[-d]$ restricted to $\iota$ are isomorphic. Hence, after reduction modulo $\pi^i$, the sheaves $p^\natural_{i+/A_i^\natural}\mathcal{P}_i$ and $\epsilon_{i*}^\natural\mathcal{O}_{S_i}[-d]$ are isomorphic in restriction to $x$, and that for all $x$.

\noindent As the sheaves $p^\natural_{+/A_i^\natural}\mathcal{P}_i$ and $\epsilon_{i*}^\natural\mathcal{O}_{S_i}[-d]$ are coherents on $A_i^\natural$ and are isomorphic on each closed point of $A_i^\natural$, Nakayama's lemma allows to claim that one has an isomorphism
\[p^\natural_{+/A_i^\natural}\mathcal{P}_i\simeq\epsilon_{i*}^\natural\mathcal{O}_{S_i}[-d],\]
and that for all $i$.
As all these isomorphisms came from a same morphism on the formal scheme $\mathcal{A}^\natural$, it has to be an isomorphism too, i.e.
\[p^\natural_{+/\mathcal{A}^\natural}\mathcal{P}\simeq\epsilon_*^\natural\mathcal{O}_\mathcal{S}[-d].\]

\end{proof}

%As explained before, because of the tensorisation by $\Ddag_{\calA/\calS,\Q}$ it is not possible to use this result to get an involutivity formula on $\calA$. However, there is an involutivity formula on $\calA^\natural$.

\begin{prop}

\[\mathcal{F}\circ\mathcal{F}^\natural\simeq\langle-1\rangle^{\natural*}\bullet[-d]:D^b_{qcoh}(\mathcal{O}_{\mathcal{A}^\natural})\rightarrow D^b_{qcoh}(\mathcal{O}_{\mathcal{A}^\natural}).\]

\end{prop}

\begin{proof}

The main result of the previous section states that
\[\mathcal{F}\circ\mathcal{F}^\natural=\Phi^{\mathcal{A}\rightarrow\mathcal{A}^\natural}_\mathcal{P}\circ\Phi^{\mathcal{A}\leftarrow\mathcal{A}^\natural}_\mathcal{P}\simeq\Phi_\mathcal{R},\]
with $\mathcal{R}=\pi_{13+/(1,3)}(\pi_{12/1}^{!(3)}\mathcal{P}\otimes\pi_{23/3}^{!(1)}\mathcal{P})$, where the $\pi_{ij}$ are the projections of $\mathcal{A}^\natural\times\mathcal{A}\times\mathcal{A}^\natural$.

\noindent In order to understand the kernel $\calR$, we will consider the sheaf
\[\calQ=\pi_{12}^*\mathcal{P}\otimes\pi_{23}^*\mathcal{P}\otimes m_{13}^*\calP^{-1},\]
with $m_{13}:\mathcal{A}^\natural\times\mathcal{A}\times\mathcal{A}^\natural\rightarrow\calA\times\calA^\natural$ being defined as $m_{13}(a,b,c)=(b,m^\natural(a,c))$.
As the kernel $\calR$ doesn't have any structure of $\calD$-module, one can forget the structure of $\calD$-module of $\calP$. Remember also that $\varphi^*\calP$ is defined as $\varprojlim\varphi_i^*\calP_i$ (there is no need to derive the functors here as the three morphisms are flat).
Using the property \ref{corresp_P}, one finds the isomorphism
\[\mathcal{Q}\simeq\pi_{12}^*(id_\calA\times \theta)^*\mathcal{P}^\vee\otimes\pi_{23}^*(id_\calA\times \theta)^*\mathcal{P}^\vee\otimes m_{13}^*(id_\calA\times \theta)^*\mathcal{P}^{\vee-1},\]
where $\calP^\vee$ is the Poincar\'e sheaf on $\calA\times\calA^\vee$. Denote by $\pi_{ij}^\vee$ the projections of $\calA^\vee\times\calA\times\calA^\vee$ and $m_{13}^\vee$ the multiplication of the first and third factors. Then
\[\mathcal{Q}\simeq(\theta\times id_\calA\times\theta)^*(\pi_{12}^{\vee*}\mathcal{P}^\vee\otimes\pi_{23}^{\vee*}\mathcal{P}^\vee\otimes m_{13}^{\vee*}\mathcal{P}^{\vee-1}).\]
As ${\calA^\vee}^\vee\simeq\calA$, $\calP^\vee$ is also the Poincar\'e sheaf on $\calA^\vee\times{\calA^\vee}^\vee$ and
\[\pi_{12}^{\vee*}\mathcal{P}^\vee\otimes\pi_{23}^{\vee*}\mathcal{P}^\vee\otimes m_{13}^{\vee*}\mathcal{P}^{\vee-1}\simeq\mathcal{O}_{\calA^\vee\times \calA\times \calA^\vee}.\]
This is just saying that $\calP^\vee$ verifies the theorem of the square for the $\calO_{\calA^\vee\times{\calA^\vee}^\vee}$-modules. Hence, as $\calO_{\calA^\natural\times\calA\times\calA^\natural}$-modules,
\[\mathcal{Q}\simeq\mathcal{O}_{\calA^\natural\times \calA\times \calA^\natural},\]
hence
\[\pi_{12}^*\mathcal{P}\otimes\pi_{23}^*\mathcal{P}\simeq m^*_{13}\calP.\]

\noindent Note that one can actually use the seesaw principle for $\calD$-modules (property \ref{D-bascule}) to prove that, as $\calD_{\calA^\natural\times \calA\times \calA^\natural/\calA^\natural\times \calA^\natural}$-modules, $\pi_{12/1}^{!(3)}\mathcal{P}\otimes\pi_{23/3}^{!(1)}\mathcal{P}\simeq m_{13/3}^{!(1)}\mathcal{P}$.

Rewriting the $\calO_{\calA^\natural\times \calA^\natural}$-module $\calR$ as $\pi_{13+/(1,3)}m_{13}^*\mathcal{P}$ and using the base change formula, one can easily find that 
%\begin{align*}
%\mathcal{R}&\simeq\pi_{13+/(1,3)}m_{13}^*\mathcal{P}
%&\simeq R\pi_{13*}(\omega_{\calA^\natural\times \calA\times \calA^\natural/\calA^\natural\times \calA^\natural}\overset{\mathbb{L}}{\otimes}_{\mathcal{D}_{\calA^\natural\times \calA\times \calA^\natural/\calA^\natural\times \calA^\natural}}m_{13}^*\mathcal{P})& \\
%& \simeq R\pi_{13*}(\pi_2^*\omega_{\calA/\calS}\overset{\mathbb{L}}{\otimes}_{\pi_2^*\mathcal{D}_{\calA/\calS}}m_{13}^*\mathcal{P})& \\
%& \simeq R\pi_{13*}(m_{13}^*p^*\omega_{\calA/\calS}\overset{\mathbb{L}}{\otimes}_{m_{13}^*p^*\mathcal{D}_{\calA/\calS}}m_{13}^*\mathcal{P})& \\
%& \simeq R\pi_{13*}m_{13}^*(\omega_{\calA\times \calA^\natural/\calA^\natural}\overset{\mathbb{L}}{\otimes}_{\mathcal{D}_{\calA\times \calA^\natural/\calA^\natural}}\mathcal{P})& \\
%& \simeq m^{\natural*}Rp_*^\natural(\omega_{\calA\times \calA^\natural/\calA^\natural}\overset{\mathbb{L}}{\otimes}_{\mathcal{D}_{\calA\times \calA^\natural/\calA^\natural}}\mathcal{P}) & \text{(using base change)}\\
%& \simeq Lm^{\natural*}p_{+/\calA^\natural}^\natural\mathcal{P}.&
%\end{align*}
$m^{\natural*}p_{+/\mathcal{A}^\natural}\mathcal{P}$.

The previous lemma allows to find an isomorphism $\mathcal{R}\simeq m^{\natural*}\epsilon_*^\natural\mathcal{O}_\mathcal{S}[-d]$. Then,
considering the commutative diagram
\[\xymatrix{
\Gamma^\natural \ar@{^{(}->}_{\tilde{\epsilon}^\natural}[d] \ar^{\tilde{m}^\natural}[r] & S \ar@{^{(}->}^{\epsilon^\natural}[d]\\
A^\natural\times A^\natural \ar_{m^\natural}[r] & A^\natural
}\]
where $\Gamma^\natural$ is the anti-diagonal of $\mathcal{A}^\natural$ (i.e. the elements of the form $(a,-a)$),
another base change leads to the isomorphism $\mathcal{R}\simeq \tilde{\epsilon}^\natural_*\mathcal{O}_{\Gamma^\natural}[-d]$.

Hence, the composition $\mathcal{F}\circ\mathcal{F}^\natural$ is the Fourier-Mukai transform of kernel $\tilde{\epsilon}^\natural_*\mathcal{O}_{\Gamma^\natural}[-d]$. The projection formula then states that 
\[\mathcal{F}\circ\mathcal{F}^\natural(\mathcal{E}^\cdot)\simeq(p_1\circ\tilde{\epsilon}^\natural)_*(p_2\circ\tilde{\epsilon}^\natural)^*\mathcal{E}^\cdot[-d],\]
with $p_1,p_2:\mathcal{A}^\natural\times\mathcal{A}^\natural\rightarrow\mathcal{A}^\natural$ the projections. As there is a commutative diagram of isomorphisms
\[\xymatrix{
&\Gamma^\natural\ar_{p_1\circ\tilde{\epsilon}^\natural}[ldd]\ar^{p_2\circ\tilde{\epsilon}^\natural}[rdd] & \\
&&\\
\mathcal{A}^\natural\ar_{\langle-1\rangle^\natural}[rr]&&\mathcal{A}^\natural
}\]
one finally concludes that
\[\mathcal{F}\circ\mathcal{F}^\natural(\mathcal{E}^\cdot)\simeq\langle-1\rangle^*\mathcal{E}^\cdot[-d].\]

\end{proof}

The following result is a direct consequence.

\begin{theo}\label{Dhz_equiv}

$\mathcal{F}$ is essentially surjective and $\mathcal{F}^\natural$ is faithfull.

\end{theo}

		\nocite{Rot96_err}

	\bibliographystyle{alpha}
	\bibliography{Biblio}

\begin{thebibliography}{{Sta}22}

\bibitem[Ber96]{Coh_rig_1}
Pierre Berthelot.
\newblock Cohomologie rigide et cohomologie rigide à supports propres {I}.
\newblock \url{https://perso.univ-rennes1.fr/pierre.berthelot/publis}, 1996.

\bibitem[Ber02]{D_mod1}
Pierre Berthelot.
\newblock Introduction \`a la th\'eorie arithm\'etique des
  $\mathcal{D}$-modules.
\newblock (279), 2002.

\bibitem[Car09]{Caro09}
Daniel Caro.
\newblock $\mathcal{D}$-modules arithm\'etiques surholonomes.
\newblock {\em Annales scientifiques de l'\'Ecole Normale Sup\'erieure},
  42:141--192, 2009.

\bibitem[Col98]{Coleman}
Robert~F. Coleman.
\newblock Duality for the de rham cohomology of an abelian scheme.
\newblock {\em Annales de l'Institut Fourier}, 48(5):1379--1393, 1998.

\bibitem[HTT08]{hotta}
Ryoshi Hotta, Kiyoshi Takeuchi, and Toshiyuki Tanisaki.
\newblock {\em D-Modules, Perverse Sheaves, and Representation Theory}.
\newblock Birkhäuser Boston, 2008.

\bibitem[Huy06]{Huy}
Daniel Huybrechts.
\newblock {\em Fourier-Mukai transforms in algebraic geometry}.
\newblock Oxford University Press, 2006.

\bibitem[Ked10]{Kedlaya}
Kiran Kedlaya.
\newblock {\em {$p$}-adic differential equations}.
\newblock Cambridge Univ. Press, 2010.

\bibitem[Lau96]{Lau}
Gérard Laumon.
\newblock Transformation de {Fourier} généralisée.
\newblock \url{https://arxiv.org/abs/alg-geom/9603004}, 1996.

\bibitem[MM74]{Ma-Me}
Barry Mazur and William Messing.
\newblock {\em Universal extensions and one dimensional crystalline
  cohomology}.
\newblock Springer-Verlag, 1974.

\bibitem[Muk81]{Muk}
Shigeru Mukai.
\newblock Duality between {$D(X)$} and {$D(\hat{X})$} with its application to
  {P}icard sheaves.
\newblock {\em Nagoya Math. J.}, 81:153--175, 1981.

\bibitem[Mum85]{Mum}
David Mumford.
\newblock {\em Abelian Varieties}.
\newblock Tata Institute of Fundamental Research, 1985.
\newblock reprint of the second edition (1974).

\bibitem[Rot96]{Rot96}
Mitchell Rothstein.
\newblock Sheaves with connection on abelian varieties.
\newblock {\em Duke Mathematical Journal}, 84(3):565--598, 1996.

\bibitem[Rot97]{Rot96_err}
Mitchell Rothstein.
\newblock Correction {to:} "{S}heaves with connection on abelian varieties".
\newblock {\em Duke Mathematical Journal}, 87(1):205--211, 1997.

\bibitem[Sch15]{Sch15}
Christian Schnell.
\newblock Holonomic $\mathcal{D}$-modules on abelian varieties.
\newblock {\em Publications Math\'{e}matiques. Institut de Hautes \'{E}tudes
  Scientifiques}, 121:1--55, 2015.

\bibitem[{Sta}22]{Sta_proj}
The {Stacks Project Authors}.
\newblock \textit{Stacks project}.
\newblock \url{https://stacks.math.columbia.edu}, 2022.

\bibitem[Vig22]{Vig22}
Florian Viguier.
\newblock Transform\'{e}e de fourier-mukai sur les schémas formels.
\newblock {\em Journal de th\'{e}orie des nombres de {B}ordeaux}, to appear,
  2022.

\bibitem[Vir04]{Virrion}
Anne Virrion.
\newblock Trace et dualité relative pour les {$\mathcal{D}$}-modules
  arithmétiques.
\newblock {\em Geometric aspects of {D}work theory}, II:1039--1112, 2004.

\end{thebibliography}

\end{document}